\documentclass[11pt,a4paper]{amsart}
\usepackage{amssymb,amsmath}
\usepackage{amsthm, dsfont, a4}
\usepackage{hyperref, enumitem}
\usepackage[all]{xy}

\usepackage[active]{srcltx}

\newcommand\comment[1]{}

\newcommand\classification[2][]{%
  \gdef\@classification{%
    \href{http://www.ams.org/msc/}%
{\textit{2000 Mathematics Subject Classification}} \ignorespaces#2\unskip}}

\leftmargin  \leftmargini
\def\NN{{\mathbb N}}

\renewcommand{\AA}{{\mathbb A}}

\def\C{{\mathcal C}}
\def\H{{\mathcal H}}
\def\G{{\mathcal G}}
\def\N{{\mathcal N}}

\def\U{{\mathcal U}}
\def\X{{\mathcal X}}
\def\O{{\mathcal O}}

\def\fT{{\mathfrak T}}
\def\fN{{\mathfrak N}}
\newcommand{\p}{{\mathfrak{p}}}
\newcommand{\q}{{\mathfrak{q}}}
\newcommand{\m}{{\mathfrak{m}}}

\def\eps{\varepsilon}
\def\id{{\rm id}}

\newcommand{\ev}{{\rm ev}}

\def\Mod{{\bf Mod}}

\newcommand{\vect}{\mathsf{vect}}
\newcommand{\Vect}{\mathsf{Vect}}
\newcommand{\alg}{\mathsf{Alg}}
\newcommand{\sets}{\mathsf{Sets}}
\newcommand{\groups}{\mathsf{Groups}}
\renewcommand{\mod}{\mathsf{Mod}}
\newcommand{\comod}{\mathsf{comod}}
\newcommand{\Comod}{\mathsf{Comod}}
\newcommand{\tenscat}[1]{\langle\!\langle#1\rangle\!\rangle}
\newcommand{\Qcoh}{\mathsf{Qcoh}}

\renewcommand{\1}{\mathds{1}}
\newcommand{\negotimes}{\!\!\otimes\!\!} 
\newcommand{\negvee}{{\!\vee}} 
\newcommand{\tp}[1]{{}^t\! #1}
\newcommand{\otimesR}{\otimes_{\!R}} 

\newcommand{\wtilde}{\widetilde}

\def\isom{\cong}

\DeclareMathOperator{\im}{Im}
\DeclareMathOperator{\Ker}{Ker}

\DeclareMathOperator{\Mor}{Mor}
\DeclareMathOperator{\Hom}{Hom}

\DeclareMathOperator{\Aut}{Aut}
\DeclareMathOperator{\End}{End}

\DeclareMathOperator{\Gal}{Gal}
\DeclareMathOperator{\Quot}{Frac}

\DeclareMathOperator{\spec}{Spec}
\DeclareMathOperator{\Sym}{Sym}
\DeclareMathOperator{\Rep}{Rep}

\def\uGal{\underline{\Gal}}

\def\markdef{\bf }

\theoremstyle{plain}
\newtheorem{thm}{Theorem}[section]
\newtheorem{cor}[thm]{Corollary}
\newtheorem{lemma}[thm]{Lemma}
\newtheorem{prop}[thm]{Proposition}

\theoremstyle{definition}
\newtheorem{defn}[thm]{Definition}
\newtheorem{exmp}[thm]{Example}
\newtheorem{rem}[thm]{Remark}

\newtheoremstyle{Acknowledgements}
  {}
    {}
     {}
     {}
    {\bfseries}
    {}
     {.5em}
     {\thmname{#1}\thmnumber{ }\thmnote{ (#3)}}
\theoremstyle{Acknowledgements}
\newtheorem{ack}{Acknowledgements.}

\begin{document}

\title[Categorical PV-theory]{A categorical approach to Picard-Vessiot theory}
\author{Andreas Maurischat}
\address{\rm {\bf Andreas Maurischat}, Lehrstuhl A f\"ur Mathematik, RWTH Aachen University, Germany }
\email{\sf andreas.maurischat@matha.rwth-aachen.de}

\classification{}

\date{\today}

\begin{abstract}
Picard-Vessiot rings are present in many settings like differential Galois theory, difference Galois theory and Galois theory of Artinian simple module algebras. In this article we set up an abstract framework in which we can prove theorems on existence and uniqueness of Picard-Vessiot rings, as well as on Galois groups corresponding to the Picard-Vessiot rings.\\
As the present approach restricts to the categorical properties which all the categories of differential modules resp.~difference modules etc.~share, it gives unified proofs for all these Galois theories (and maybe more general ones).

%
\end{abstract}

\maketitle

\section{Introduction}

Since the foundation of Picard-Vessiot theory as a Galois theory for linear differential equations (cf.~\cite{ep:ta-tome3}), many analogues have evolved. For example, Picard-Vessiot theory for difference equations \cite{mvdp-mfs:gtde}, for iterative differential equations \cite{bhm-mvdp:ideac}, for $C$-ferential fields \cite{mt:haapvt}, for Artinian simple module algebras \cite{ka-am:pveasma} and others.\\
In all these theories the base ring is a commutative ring with some operators acting on it, and the main objects are modules over that ring with the same operators acting.\\
The setting of Artinian simple module algebras generalises the setting of (iterative) differential fields as well as that of inversive difference pseudo-fields (i.e. simple difference rings which are a product of fields), but it does not generalise the difference setting where the given endomorphism is not bijective as in \cite{mw:gdgt}. Y.~Andr\'e in \cite{ya:dnctgdd}  already gave a setting which unifies the case of difference pseudo-fields and differential fields in characteristic zero, however, it doesn't contain the Picard-Vessiot theory for differentially simple rings given in \cite{am:pvtdsr}.\\
One could go further and generalise the operators even more or loosen the conditions on the base ring. However, there might still be cases not covered by such generalisations.\\
The present approach therefore restricts to the categorical properties which all the categories of differential modules resp.~difference modules etc.~share, and hence gives unified proofs for all these Picard-Vessiot theories (and more general ones).

The main results of this paper are the construction of a universal solution ring for a given ``module'' $M$ such that all Picard-Vessiot rings (PV-rings) for $M$ are quotients of this ring (Thm.~\ref{thm:exists-sol-ring} and Thm.~\ref{thm:simple-minimal-solution-rings-are-quotients}), the existence of PV-rings up to a finite extension of constants (Thm.~\ref{thm:existence-of-pv-ring}), and uniqueness of PV-rings inside a given simple solution ring with same constants (Prop.~\ref{prop:unique-pv-inside-simple-sol-ring}).
Furthermore, we prove a correspondence between isomorphism classes of fibre functors $\omega:\tenscat{M}\to \vect_{\tilde{k}}$ and isomorphism classes of PV-rings $R$ for $M\otimes_k \tilde{k}$, where $k$ is the field of constants of the base ring $S$ and $\tilde{k}$ is any finite extension of $k$ (Thm.~\ref{thm:pv-rings-equiv-to-fibre-functors}).
We also prove that the group scheme of automorphisms $\underline{\Aut}^\partial(R/S)$ of $R$ over $S$ that commute with the extra structure, is isomorphic to the affine group scheme of automorphisms $\underline{\Aut}^\otimes (\omega)$ of the corresponding fibre functor $\omega$ (Cor.~\ref{cor:auts-are-isomorphic}). These two statements are direct generalisations of the corresponding facts given for example in \cite[Ch.~9]{pd:ct} or \cite[Sect.~3.4 and 3.5]{ya:dnctgdd}.\\
Finally, we give a Galois correspondence between closed normal subgroup schemes of the Galois group scheme and subalgebras of the PV-ring which are PV-rings for some other ``module''.

\medskip

At this point we should mention that the setup of this article does not cover the parametrized Picard-Vessiot theories where the constants are equipped with an additional differential or difference operator as given for example in \cite{pjc-mfs:gtpdeldag}, \cite{ldv-ch-mw:dgtlde}, \cite{ch-mfs:dgtlde}.

\medskip

\paragraph{\bf Differential setting}
We now recall the main properties of the differential setting for having a better comparison with its analogs in the abstract setting.

Classically, one starts with some differential field $(F,\partial)$ of characteristic zero, and its field of differentially constant elements $k:=F^\partial=\{x\in F\mid \partial(x)=0\}$.

The basic objects are differential modules ($\partial$-modules) $(M,\partial_M)$, i.e.~$F$-vector spaces $M$ with a derivation $\partial_M:M\to M$. Morphisms of $\partial$-modules (called \textit{differential homomorphisms}) are homomorphisms $f:M\to N$ of the underlying $F$-vector spaces which are compatible with the derivations, i.e.~satisfy $f\circ \partial_M=\partial_N\circ f$. This implies that kernels and cokernels of $\partial$-homomorphisms are again $\partial$-modules, turning the category of $\partial$-modules over $(F,\partial)$ into an abelian category.\\
For $\partial$-modules $(M,\partial_M)$ and $(N,\partial_N)$ the tensor product $M\otimes_F N$ is naturally equipped with a derivation given by
$\partial(m\otimes n):=\partial_M(m)\otimes n+m\otimes \partial_N(n)$. This provides the category of $\partial$-modules with the structure of a symmetric monoidal category with unit object $\1$ given by the differential field $(F,\partial)$. 
Furthermore, for every $\partial$-module $(M,\partial_M)$ that is finitely generated as an $F$-vector space the dual vector space $M^\vee$ carries a differential structure $\partial_{M^\vee}$ such that the natural homomorphisms of evaluation $\ev:M\otimes M^\vee\to F$ and coevaluation $\delta:F\to M^\vee\otimes M$ are $\partial$-homomorphisms. This means that $(M^\vee, \partial_{M^\vee})$ is a dual of $(M,\partial_M)$ in the category of $\partial$-modules.\\
As we consider all $\partial$-modules -- and not only those which are finitely generated as $F$-vector spaces -- this category is even closed under inductive limits. This is due to the fact that for a directed system $(M_i,\partial_i)_{i\in I}$ of differential modules, the inductive limit
$\varinjlim_{i\in I} M_i$ of $F$-vector spaces can be equipped uniquely with a derivation compatible with the homomorphisms 
$M_i\to \varinjlim_{i\in I} M_i$.

The differential constants of a $\partial$-module  $(M,\partial_M)$ are given as
$M^\partial:=\{m\in M\mid \partial_M(m)=0\}$. This is a $k$-vector space of dimension at most $\dim_F(M)$. Therefore, one is interested in differential field extensions of $F$ over which the corresponding dimensions are the same. From the view of linear differential equations this means that the differential field extension contains a full set of solutions.

We assume now that the field of constants $k$ is algebraically closed. 
A Picard-Vessiot extension of $F$ for a $\partial$-module  $(M,\partial_M)$ with $\dim_F(M)<\infty$ is defined to be a minimal differential field extension $(E,\partial_E)$ of $F$ such that $\dim_k((E\otimes_F M)^\partial)=\dim_E(E\otimes_F M)=\dim_F(M)$. A main theorem states that a Picard-Vessiot extension always exists and is unique up to differential isomorphism.\\
The differential Galois group $\Gal(E/F)$ of a Picard-Vessiot extension $E/F$ is then defined to be the group $\Aut^\partial(E/F)$ of differential automorphisms of $E$ fixing $F$. It has the structure of ($k$-rational points of) a  linear algebraic group over $k$, and one obtains a Galois correspondence between the Zariski-closed subgroups of $\Gal(E/F)$ and differential subfields of $E$ containing $F$.\\
A main role is played by the Picard-Vessiot ring $R$ in $E$. It is the subring of $E$ which is generated as an $F$-algebra by the entries of a fundamental solution matrix and its inverse\footnote{A fundamental solution matrix is a base change matrix over $E$ mapping an $F$-basis of $M$ to a $k$-basis of $(E\otimes_F M)^\partial$, both bases seen as $E$-bases of $E\otimes_F M$.}. $R$ is a $\partial$-simple $\partial$-ring extension of $F$ minimal with the property that $R\otimes_F M$ has a basis of constant elements. Here, $\partial$-simple means that $R$ has no nontrivial ideals stable under the derivation. Furthermore, $E$ is the field of fractions of $R$, and $\Aut^\partial(R/F)=\Aut^\partial(E/F)$. Moreover, the spectrum $\spec(R)$ is a torsor of $\Gal(E/F)$ over $F$.
The Galois correspondence is more or less a consequence of this torsor property, as the subfield $E^\H$ corresponding to a closed subgroup $\H\leq \Gal(E/F)$ is nothing else than the field of rational functions on the scheme $\spec(R)/\H$.

If the field of constants $k$ is not algebraically closed (cf.~\cite{td:tipdgtfrz} and \cite{am:gticngg}), some things become more envolved. First at all, one also requires that a Picard-Vessiot field $E$ has the same field of constants $k$ -- a condition which is automatically fulfilled if $k$ is algebraically closed.
Furthermore, the Galois group has to be replaced by a representable group functor $\uGal(E/F)$, i.e.~an affine group scheme, 
whose group of $k$-rational points is $\Aut^\partial(E/F)$. Then as above, $\spec(R)$ is a $\uGal(E/F)$-torsor over $F$ and one obtains a Galois correpondence between closed subgroups of $\uGal(E/F)$ and differential subfields of $E$ containing $F$.
However, since the constants are not algebraically closed, existence of a Picard-Vessiot field or a Picard-Vessiot ring is not guaranteed, and also uniqueness might fail. Furthermore, assume one is given a PV-field $E$, the Galois group scheme does not act algebraically on the PV-field but only on the PV-ring. On the other hand, one does not get a full Galois correspondence on the ring level. The geometric reason is that for a closed subgroup $\H\leq \uGal(E/F)$ the invariant ring $R^\H$ is the ring of global sections of the orbit space $\spec(R)/\H$. If the latter is not affine, $R^\H$ becomes ``too small''.\\
On the ring level, at least one has a restricted Galois correspondence between closed normal subgroups of $\uGal(E/F)$ and differential subrings of $R$ containing $F$ which are Picard-Vessiot rings for some $\partial$-module (cf.~\cite{am:pvtdsr}). 

In the abstract setting of this article, we will stay on the ring level, since the action of the Galois group is naturally algebraic there.

\medskip

\paragraph{\bf Iterative differential and difference setting}

In iterative differential Galois theory in arbitrary characteristic derivations are replaced by so called iterative derivations (cf.~\cite{bhm-mvdp:ideac}). These are a collection $\theta=\left( \theta^{(n)}\right)_{n\in\NN}$ of additive maps satisfying $\theta^{(0)}=\id$, $\theta^{(n)}(ab)=\sum_{i+j=n}\theta^{(i)}(a)\theta^{(j)}(b)$ as well as $\theta^{(n+m)}=\binom{n+m}{n}\theta^{(n)}\circ \theta^{(m)}$ for all $n,m\in \NN$. This means, $\partial:=\theta^{(1)}$ is a derivation and $\theta^{(n)}$ resembles $\frac{1}{n!}\partial^n$ -- the $n$-th iterate of $\partial$ devided by $n$-factorial. 
Indeed, in characteristic zero, the iterative derivations are determined by the derivation $\partial=\theta^{(1)}$ via $\theta^{(n)}=\frac{1}{n!}\partial^n$. In particular the differential setting in characteristic zero is a special case of the iterative differential setting.
The constants of an iterative differential field $(F,\theta)$ are given by 
$F^\theta:=\{x\in F\mid \theta^{(n)}(x)=0 \, \forall n\geq 1 \}$. The basic objects are iterative differential modules $(M,\theta_M)$, and one is interested in minimal iterative differential extensions $E$ of $F$ (with same constants) such that 
$\dim_{F^\theta}\left( (E\otimes_F M)^\theta \right)=\dim_F(M)$. All the things about Picard-Vessiot rings and fields turn out the same as in the differential setting. However, even in the case that $k=F^\theta$ is algebraically closed, one has to consider the Galois group as an affine group scheme which might be nonreduced (if $E/F$ is not separable) (cf.~ \cite{am:igsidgg}, \cite{am:gticngg}).

\smallskip

In difference Galois theory derivations are replaced by automorphisms and constants by invariants, i.e.~one starts with some field $F$ together with an automorphism $\sigma:F\to F$ and its field of invariant elements $k:=F^\sigma:=\{ x\in F\mid \sigma(x)=x\}$. The basic objects are difference modules $(M,\sigma_M)$, i.e.~$F$-vector spaces $M$ together with a $\sigma$-linear automorphism $\sigma_M:M\to M$. Again, the set of invariants
$M^\sigma:=\{ m\in M\mid \sigma_M(m)=m\}$ is a $k$-vector space of dimension at most $\dim_F(M)$, and one is interested in a difference extension of $F$ over which the corresponding dimensions are the same. In this setting another aspect appears, since in some situations every solution ring has zerodivisors. Hence even if $k$ is algebraically closed, there does not exist a Picard-Vessiot {\bf field} in general.
Nevertheless, if $k$  is algebraically closed, there always exists a Picard-Vessiot ring $R$ over $F$, i.e.~a $\sigma$-simple $\sigma$-ring extension $R$ of $F$ minimal with the property that $R\otimes_F M$ has a basis of invariant elements,
and instead of the Picard-Vessiot field one considers $E=\Quot(R)$, the total ring of fractions of $R$.
With these definitions one again obtains a Galois group scheme $\uGal(R/F)$ as a representable functor whose $k$-rational points are exactly $\Aut^\sigma(R/F)=\Aut^\sigma(E/F)$, as well as a Galois correspondence between closed subgroup schemes of $\uGal(R/F)$ and total difference subrings of $E$ containing $F$.

\medskip

\paragraph{\bf Other settings}

The three basic settings described above have been generalised in various ways. First at all, the operators acting have become more general: Takeuchi in \cite{mt:haapvt} considered an action of a pointed irreducible cocommutative coalgebra $C$ on the base field $F$ (which he then calls a $C$-ferential field). This amounts to having a collection of several commuting higher derivations.
Later Amano-Masuoka in \cite{ka-am:pveasma} have considered an action of a pointed cocommutative Hopf-algebra $D$ on the base field $F$ (then called $D$-module algebra), though generalising to a collection of commuting iterative derivations and automorphisms.
Andr\'e in \cite{ya:dnctgdd} used so called noncommutative differentials in characteristic $0$ resembling a collection of derivations and endomorphisms.

On the other hand, also the bases have become more general: the base field $F$ has been generalised to (i) an Artinian algebra (i.e. finite product of fields) which is simple as $D$-module algebra in \cite{ka-am:pveasma}, (ii) a Noetherian ring which is simple with respect to the differentials in  \cite{ya:dnctgdd}, and (iii) any differentially simple (iterative) differential ring in \cite{am:pvtdsr}.

In \cite[Ch.~2]{nk:esde}, N.~Katz even considers schemes $\X$ of finite type over $k$, and obtains Picard-Vessiot extensions for finitely generated $\O_\X$-modules with integrable connections.

\medskip

All these settings have in common that you start with a base ring (or even base scheme) $F$ with some extra structure such that no non-trivial ideal of $F$ is respected by the extra structure, i.e.~that $F$ is simple. The basic objects for which one considers Picard-Vessiot rings are finitely generated modules over $F$ with corresponding extra structure having a dual in the category
of modules with extra structure, and the Picard-Vessiot rings are algebra objects in the category of (all) modules with extra structure.

\medskip

\paragraph{\bf Abstract setting}

In the abstract setting this is reflected by the following basic setup:
\begin{enumerate}
\item[(C1)]  $\C$ is an abelian symmetric monoidal category with unit object $\1\in \C$. We assume that $\1$ is a simple object in $\C$.
\item[(C2)] $\C$ is cocomplete, i.e.~$\C$ is closed under small inductive limits.
\item[(F1)] There is a scheme $\X$, and an additive tensor functor $\upsilon:\C\to \Qcoh(\X)$ from $\C$ to the category of quasi-coherent $\O_\X$-modules which is faithful, exact and preserves small inductive limits. (In particular, $\upsilon(\1)=\O_X$.)
\item[(F2)] $M\in \C$ is dualizable whenever $\upsilon(M)$ is a finitely generated $\O_\X$-module.
\end{enumerate}
It is this basic setup from which all the statements on Picard-Vessiot rings and their Galois groups follow. For stating those, one has to transfer several concepts into the abstract setting; most important the concept of constants/invariants: \\
It is not hard to see that for every differential module $(M,\partial_M)$ over $F$  the constants $M^\partial$ of $M$ can also be given as the vector space $\Hom_F^\partial(F,M)$ of differential homomorphisms $f:F\to M$, since every $F$-homomorphism $f:F\to M$ is uniquely determined by the image of $1\in F^\partial\subseteq F$. Similarly, the invariants $M^\sigma$ of a difference module $(M,\sigma_M)$ can be given as $\Hom_F^\sigma(F,M)$. Hence, in the abstract setting, ``taking constants'' is given by the functor
$()^\C:=\Mor_\C(\1,-):\C\to \Vect_k$ where $k$ is the field $k=\End_\C(\1)$ corresponding to the constants of a differential field $F$ resp.~the invariants of a difference field $F$.

The condition on a Picard-Vessiot ring $R$ for $M$ that the module $R\otimes_F M$ has a basis of constants/invariants is given abstractly by the condition that the natural morphism $\eps_{R\otimes M}:R\otimes \iota\left((R\otimes M)^\C\right) \to R\otimes M$ is an isomorphism
in the category $\C$ (cf.~Prop.~\ref{prop:on-iota-r}). Here $\iota:\Vect_k\to \C$ is a functor corresponding to the construction of a differential/difference module out of a $F^\partial$-vector space by tensoring with the base differential/difference ring $F$.

\medskip

The article is structured as follows.
In Section \ref{sec:comm-alg-thm}, we prove a theorem on commutative algebras which will later be used for showing that the constants of minimal simple solution rings are just a finite extension of the constants $k$, and in particular guarantee the existence of Picard-Vessiot rings up to a finite extension of constants.

In Section \ref{sec:setup}, we investigate some properties of the functors $()^\C$ and $\iota$. In particular, we show that the functor $()^\C$ is right adjoint to $\iota$. Furthermore, we show that the unit $\eta:\id_{\Vect_k}\to ()^\C\circ \iota$ of the adjunction is a natural isomorphism, and that the counit $\eps:\iota\circ ()^\C\to \id_\C$ of the adjunction provides a monomorphism $\eps_M$ for every $M\in \C$. The latter corresponds to the fact in the differential setting that the natural homomorphism $F\otimes_k M^\partial \to M$ is injective.

Section \ref{sec:c-algebras} is dedicated to commutative algebras $R$ in the category $\C$ and the category $\C_R$ of $R$-modules in $\C$ as given in \cite{sml:ca}, as well as properties of the functors $\iota_R$ and $()^{\C_R}$ similar to those of $\iota$ and $()^\C$, under certain assumptions on the algebra $R$.

Solution rings and Picard-Vessiot rings are then the subject of Section \ref{sec:solution-rings}, where also the theorems on existence and uniqueness of Picard-Vessiot rings are proven.

The objective of Section \ref{sec:pv-rings-and-fibre-functors} is the correspondence between isomorphism classes of Picard-Vessiot rings for a given dualizable $M\in \C$ and isomorphism classes of fibre functors from the strictly full abelian tensor subcategory $\tenscat{M}$ of $\C$ to $\Vect_k$.

In Section \ref{sec:galois-groups} we consider the group functors
$\underline{\Aut}_{\C-\text{alg}}(R)$ of automorphisms of $R$ and  $\underline{\Aut}^\otimes(\omega_R)$ of automorphisms of the corresponding fibre functor $\omega_R$, and we show that they are both isomorphic to the spectrum of the $k$-algebra $\omega_R(R)=(R\otimes R)^\C$. As the latter will be proven to be a Hopf-algebra of finite type over $k$, both group functors are indeed affine group schemes of finite type over $k$.

Finally, in Section \ref{sec:galois-correspondence} we prove the Galois correspondence between normal closed subgroups of the Galois group scheme $\underline{\Aut}_{\C-\text{alg}}(R)$ and $\C$-subalgebras of $R$ that are Picard-Vessiot rings for some dualizable $N\in \C$.

\begin{ack}
I would like to thank G.~B\"ockle and F.~Heiderich for their comments on earlier versions which helped a lot to improve the paper. I would also like to thank M.~Wibmer, as only a common project with him drew my attention to this general abstract setting.
\end{ack}

\section{A commutative algebra theorem}\label{sec:comm-alg-thm}

We will be faced with the question whether there exists a Picard-Vessiot ring up to a finite extension of constants. The following theorem will be a key incredient to the existence proof.
All algebras are assumed to be commutative with unit.

\begin{thm}\label{thm:abstract-algebra}
Let $k$ be a field, $S$ an algebra over $k$ and $R$ a finitely generated flat $S$-algebra. 
Furthermore, let $\ell$ be a  field extension of $k$ such that $S\otimes_k \ell$ embeds into $R$ as an $S$-algebra. Then $\ell$ is a finite extension of $k$.
\end{thm}

\begin{proof}

The proof is split in several steps:

\textbf{1) Reduction to $S$ being a field} \ \\
Choose a minimal prime ideal $\p$ of $S$, and let $S_\p$ denote the localization of $S$ at $\p$. Since localizations are flat, the inclusion of rings
$S\subseteq S\otimes_k \ell\subseteq R$ induces an inclusion of rings
$$S_\p\subseteq S_\p\otimes_k \ell\subseteq S_\p\otimes_S R,$$
and $S_\p\otimes_S R$ is a finitely generated $S_\p$-algebra. Since flatness is stable under base change, $S_\p\otimes_S R$ is a flat $S_\p$-algebra.\\
Since $\p S_\p$ is the maximal ideal of $S_\p$, $\bar{S}:=S_\p/\p S_\p$ is a field, and $\bar{R}:=S_\p/\p S_\p\otimes_S R$  is a finitely generated flat algebra over $\bar{S}$.
It remains to show that $\bar{S}\otimes_k \ell$ embeds into $\bar{R}$.

 Since $S_\p\otimes_k \ell$ and $S_\p\otimes_S R$ are both flat over $S_\p$, the exact sequence 
$0\to \p S_\p\to S_\p\to S_\p/\p S_\p\to 0$ leads to a commutative diagram with exact rows 
$$\xymatrix{
0 \ar[r] & \p S_\p\otimes_k \ell \ar[r] \ar@{^{(}->}[d] & S_\p\otimes_k \ell \ar[r] \ar@{^{(}->}[d] & \left(S_\p/\p S_\p\right)\otimes_k \ell \ar[r] \ar[d] & 0 \\
0 \ar[r] & \p S_\p\otimes_S R\ar[r] & S_\p\otimes_S R\ar[r] & \left( S_\p/\p S_\p\right)\otimes_S R\ar[r] & 0.
}$$
Then the last vertical arrow is an injection if the left square is a pullback diagram. Hence, we have to proof that any element in $S_\p\otimes_k \ell$ whose image in $S_\p\otimes_S R$ actually lies in $\p S_\p\otimes_S R$ is an element of $ \p S_\p\otimes_k \ell$.

Hence, let $z=\sum_{i=1}^n s_i\otimes x_i\in S_\p\otimes_k \ell$ with $k$-linearly independent $x_1,\dots, x_n\in \ell$,  and let $w=\sum_{j=1}^m a_j\otimes r_j\in \p S_\p\otimes_S R$ such that their images in $S_\p\otimes_S R$ are the same. Since all elements in $\p S_\p$ are nilpotent, there is $e_1\geq 0$ maximal such that $a_1^{e_1}\ne 0$. Inductively for $j=2,\dots, m$, there is $e_j\geq 0$ maximal such that 
$a_1^{e_1}\cdots a_j^{e_j}\ne 0$. Let $a:=\prod_{j=1}^m a_j^{e_j}\in S_\p$. Then by construction, $a\ne 0$ but $a\cdot w=\sum_{j=1}^m a a_j\otimes r_j = 0$. So $0=a\cdot z=\sum_{i=1}^n as_i\otimes x_i$, i.e.~$as_i=0$ for all $i$. Since $a\ne 0$, one obtains $s_i\not\in (S_\p)^\times$, i.e.~$s_i\in \p S_\p$.

\medskip

From now on, we may and will assume that $S$ is a field. In this case $R$ is Noetherian as it is a finitely generated $S$-algebra.

\textbf{2) Proof that $\ell$ is algebraic over $k$}

Assume that $\ell$ is not algebraic over $k$, then there is an element $a\in \ell$ transcendental over $k$. By assumption, $a$ is also transcendental over $S$ inside $R$, i.e.~the polynomial ring $S[a]$ is a subring of $R$.
The image of the corresponding morphism $\psi:\spec(R)\to \spec(S[a])
\cong \AA_S^1$ is a dense subset of $\spec(S[a])$, since the ringhomomorphism is an inclusion, and it is locally closed by \cite[Cor.~3, Ch.~V, \S 3.1]{nb:cac1-7}.
Hence, the image is open. But for all $0\ne f\in k[a]$, the irreducible factors of $f$ in $S[a]$, are invertible in $\ell\subseteq R$. Hence, infinitely many maximal ideals of $\spec(S[a])$ are not in the image of $\psi$ -- contradicting that the image is open.

\textbf{3) Proof that $\ell$ is finite over $k$}

For showing that $\ell$ is indeed finite over $k$, we give a bound on $[\ell':k]$ for any $\ell'\subseteq \ell$ which is finite over $k$, and this bound only depends on data of $R$.
Since $\ell$ is the union of all its finite subextensions this proves finiteness of $\ell$.

For simplicity we again write $\ell$ for the finite extension $\ell'$ of $k$.\\
Let 
$$(0)=\bigcap_{i=1}^c \q_i$$
be a primary decomposition of the zero ideal $(0)\subseteq R$ and $\p_i:=\sqrt{\q_i}$ the corresponding prime ideals. Furthermore, let $N_i\in \NN$ satisfy $\p_i^{N_i}\subseteq \q_i$, i.e.~for all $y_1,\dots, y_{N_i}\in \p_i$, one has $y_1\cdot y_2\cdots y_{N_i}\in \q_i$.\footnote{This $N_i$ exists since $R$ is Noetherian and therefore $\p_i$ is finitely generated.}
Furthermore, for each $i=1,\dots, c$ let $\m_i\subseteq R$ be a maximal ideal containing $\p_i$. Then $d_i:=\dim_S{R/\m_i}$ is finite for all $i$.

We claim that $\dim_k(\ell)$ is bounded by $2\cdot \sum_{i=1}^c d_i\cdot N_i$:

First at all $R\to \prod_{i=1}^c R/\q_i$ is an injective $S$-algebra homomorphism and $R/\q_i$ is irreducible with unique minimal ideal $\p_i$.\\
Letting $\tilde{\q_i}:=\q_i\cap (S\otimes_k \ell)$, and $\tilde{\p_i}:=\p_i\cap  (S\otimes_k \ell)=\sqrt{\tilde{\q_i}}$, then $(S\otimes_k \ell)/\tilde{\q_i}$ embeds into $R/\q_i$, and 
$S\otimes_k \ell \to \prod_{i=1}^c (S\otimes_k \ell)/\tilde{\q_i}$ is injective. It therefore suffices to show that $\dim_S\left( (S\otimes_k \ell)/\tilde{\q_i} \right)\leq 2d_iN_i$ holds for each $i$.
In the following we therefore consider an arbitrary component and will omit the index $i$. 

Since $(S\otimes_k \ell)/\tilde{\q}$ is a finite $S$-algebra, and $\tilde{\p}$ is its unique minimal prime ideal, $(S\otimes_k \ell)/\tilde{\q}$ is a local Artinian algebra with residue field $(S\otimes_k \ell)/\tilde{\p}$. Since $(S\otimes_k \ell)/\tilde{\p}$ is a field, the composition
$$(S\otimes_k \ell)/\tilde{\p} \hookrightarrow R/\p \to R/\m$$
is injective. Hence, $$\dim_S\left( (S\otimes_k \ell)/\tilde{\p}\right)\leq \dim_S\left( R/\m\right)=d.$$
It remains to show that $\dim_{(S\otimes_k \ell)/\tilde{\p}}\left( (S\otimes_k \ell)/\tilde{\q} \right)\leq 2N$.

As a tensor product of fields and as $\ell/k$ is finite, $S\otimes_k \ell$ is a finite direct product of local artinian algebras with residue fields being finite extensions of $S$. The  local algebra over some finite extension $S'$ of $S$ is given as $S'\otimes_{k'} \tilde{k}$ for a finite extension $k'$ of $k$ contained in $S'$ and a purely inseparable extension $\tilde{k}/k'$.

In particular, also the algebra $(S\otimes_k \ell)/\tilde{\q}$ is of that form (as it is just isomorphic to one factor of $(S\otimes_k \ell)$). Hence, let $S'$, $k'$ and $\tilde{k}$ be such that $(S\otimes_k \ell)/\tilde{\p}\cong S'$ and
$(S\otimes_k \ell)/\tilde{\q}\cong S'\otimes_{k'} \tilde{k}$.
As $\tilde{k}$ is purely inseparable over $k'$, there are $x_1,\dots, x_t\in \tilde{k}$, $m_1,\dots, m_t\in \NN$ and $a_1,\dots, a_t\in k'$ such that 
$$\tilde{k}=k'[x_1,\dots, x_t]/\left(x_1^{p^{m_1}}-a_1,\dots, x_t^{p^{m_t}}-a_t\right).$$
where $p$ denotes the characteristic of the fields.
As $S'\otimes_{k'} \tilde{k}$ is local with residue field $S'$, there are also $s_1,\dots, s_t\in S'$ such that $s_j^{p^{m_j}}=a_j$ for all $j=1,\dots, t$, and $S'\otimes_{k'} \tilde{k}$ is given as
$$S'\otimes_{k'} \tilde{k}\cong S'[x_1,\dots, x_t]/\left((x_1-s_1)^{p^{m_1}},\dots, (x_t-s_t)^{p^{m_t}}\right).$$
In particular its nilradical (corresponding to $\tilde{\p}$) is generated by $(x_1-s_1,\dots, x_t-s_t)$.

Since $\tilde{\p}^N\subseteq \tilde{\q}$, and $(x_1-s_1)^{p^{m_1}-1}\cdots (x_t-s_t)^{p^{m_t}-1}\ne 0$ we obtain that
$$N>\sum_{j=1}^t (p^{m_j}-1)\geq \sum_{j=1}^t \frac12 p^{m_j}= \frac12 \dim_{S'}(S'\otimes_{k'} \tilde{k}).$$

Therefore, we have shown that $\dim_{(S\otimes_k \ell)/\tilde{\p}}\left( (S\otimes_k \ell)/\tilde{\q} \right)<2N$.

\end{proof}

\section{Setup and basic properties}\label{sec:setup}

In this section, we set up an abstract framework in which we can prove theorems on Picard-Vessiot extensions, as well as their Galois groups. The theorems thus apply to all kinds of differential and difference Galois theories which match the basic setup given below. The setup therefore provides a uniform approach to the existing theories.

\medskip

We consider the following setup:

\begin{enumerate}
\item[(C1)]  $\C$ is a locally small abelian symmetric monoidal category with unit object $\1\in \C$. We assume that $\1$ is a simple object in $\C$.
\item[(C2)] $\C$ is cocomplete, i.e.~$\C$ is closed under small inductive limits.
\item[(F1)] There is a scheme $\X$, and an additive tensor functor $\upsilon:\C\to \Qcoh(\X)$ from $\C$ to the category of quasi-coherent $\O_\X$-modules which is faithful, exact and preserves small inductive limits. (In particular, $\upsilon(\1)=\O_\X$.)
\item[(F2)] $M\in \C$ is dualizable whenever $\upsilon(M)$ is a finitely generated $\O_\X$-module.
\end{enumerate}

\begin{rem}
\begin{enumerate}
\item The presence of a faithful functor $\upsilon:\C\to \Qcoh(\X)$ as stated in (F1) already implies that all $\Mor_{\C}(M,N)$ are abelian groups, i.e.~that $\C$ is locally small. Hence, we could have ommitted this condition in (C1). However, in this section and Section \ref{sec:c-algebras}, we will not use conditions (F1) and (F2) and therefore need the condition ``locally small'' in (C1).
\item By an object $M\in \C$ being  \textit{dualizable}, we mean that $M$ admits a (right) dual, i.e.~an object $M^\vee\in \C$ together with two morphisms $\ev_M:M\otimes M^\vee\to \1$ (\textit{evaluation}) and
$\delta_M:\1\to M^\vee \otimes M$ (\textit{coevaluation}) such that the diagrams
\[ \xymatrix@C+6pt{ M^\vee\cong \1\negotimes M^\vee \ar[r]^(.53){\delta_M\otimes \id_{M^\vee} } \ar[dr]_{\id_{M^\vee}}& 
M^\vee \negotimes M\negotimes M^\vee \ar[d]^{\id_{M^\vee}\otimes \ev_M} \\
 & M^\vee \negotimes \1\cong M^\vee
}
\quad \text{and} \quad
 \xymatrix{ M\cong M \negotimes \1 \ar[r]^{\id_M\otimes \delta_M} \ar[rd]_{\id_M} &
 M\negotimes M^\vee \negotimes M \ar[d]^{\ev_M\otimes \id_M} \\
& \1 \negotimes M\cong M
}
\]
commute. 
\end{enumerate}
\end{rem}


\medskip

\begin{exmp}
All the settings mentioned in the introduction are examples for the category $\C$.
\end{exmp}

\medskip


In the remainder of this section, $\C$ will be a category satisfying properties (C1) and (C2).

Let $k:=\End_\C(\1)$ denote the ring of endomorphisms of the unit object $\1$. Then by simplicity of $\1$, $k$ is a division ring, and even a field, as $\End_\C(\1)$ is always commutative.

Let $\Vect_k$ denote the category of $k$-vector spaces, and $\vect_k$ the subcategory of finite dimensional $k$-vector spaces. 
There is a functor $\otimes_k:\C\times \vect_k\to \C$ such that $M\otimes_k k^n=M^n$ and in general $M\otimes_k V\cong M^{\dim(V)}$ (cf.~\cite{pd-jsm:tc}
, p.~21 for details). \\
As $\C$ is cocomplete, the functor $\otimes_k$ can be extended to
$\otimes_k:\C\times \Vect_k\to \C$ via
$$M\otimes_k V :=\varinjlim\limits_{\substack{W\subset V\\ \text{fin.dim.}}} M\otimes_k W.$$
This functor fulfills a functorial isomorphism of $k$-vector spaces
$$\Mor_\C(N,M\otimes_k V) \cong  \Mor_\C(N,M) \otimes_k V\text{ for all } M,N\in \C, V\in \Vect_k,$$ 
where the tensor product on the right hand side is the usual tensor product of $k$-vector spaces. Recall that $\Mor_\C(N,M)$ is a $k$-vector space via the action of $k=\End_\C(\1)$.

The functor $\otimes_k$ induces a tensor functor $\iota:\Vect_k\to \C$ given by $\iota(V):=\1\otimes_k V$, and one obviously has $M\otimes_k V\cong M\otimes \iota(V)$ (the second tensor product taken in $\C$). The functor $\iota$ is faithful and exact by construction.
Since $\iota$ is an exact tensor functor and all finite dimensional vector spaces have a dual (in the categorial sense), all objects $\iota(V)$ for $V\in \vect_k$ are dualizable in $\C$.

There is also a functor $(-)^\C:=\Mor_\C(\1, -):\C\to \Vect_k$ from the category $\C$ to the category of all $k$-vector spaces.

\begin{rem}
As already mentioned in the introduction, in the differential case $M^\C=M^\partial$ is just the $k$-vector space of constants of the differential module $M$. In the difference case (with endomorphism $\sigma$), $M^\C$ equals the invariants $M^\sigma$ of the difference module $M$.\\
The functor $\iota$ corresponds to the construction of ``trivial'' differential (resp.~difference) modules by tensoring a $k$-vector space with the differential (resp.~difference) base field $F$.
\end{rem}

The following proposition gives some properties of the functors $\iota$ and $(-)^\C$ which are well known for differential resp.~difference modules.



\begin{prop}\label{prop:first-properties}
Let $\C$ be a category satisfying (C1) and (C2), and let $\iota$ and $()^\C$ be as above.
Then the following hold.
\begin{enumerate}
\item If $V\in \Vect_k$, then any subobject and any quotient of $\iota(V)$ is isomorphic to $\iota(W)$ for some $W\in \Vect_k$. 
\item If $V\in \vect_k$, then $\iota(V)\in \C$ has finite length and ${\rm length}(\iota(V))=\dim_k(V)$.
\item If $M\in \C$ has finite length, then $M^\C\in \vect_k$ and $\dim_k(M^\C)\leq {\rm length}(M)$.
\end{enumerate}
\end{prop}

\begin{proof} 
\begin{enumerate}[leftmargin=*, widest*=3]
\item First consider the case that $V\in \Vect_k$ is of finite dimension. We show the claim by induction on $\dim(V)$.\\ The case $\dim(V)=0$ is trivial.
Let $V\in \vect_k$ and $N\in \C$ be a subobject of $\iota(V)$, and let $V'\subseteq V$ be a $1$-dimensional subspace. Then one has a split exact sequence of $k$-vector spaces $0\to V'\to V\to V/V' \to 0$ and therefore a split exact sequence
$$0\to \iota(V')\to \iota(V)\to \iota(V/V')\to 0$$ in $\C$.
Since $N$ is a subobject of $\iota(V)$, the pullback $N\cap \iota(V')$ is a subobject of $\iota(V')\cong \1$. As $\1$ is simple, $N\cap \iota(V')=0$ or $N\cap \iota(V')=\iota(V')$.\\
In the first case, $N\hookrightarrow \iota(V/V')$, and the claim follows by induction on $\dim(V)$.\\
In the second case, by induction $N/\iota(V')$ is isomorphic to $\iota(W)$ for some subspace $W \subseteq  \iota(V/V')$.
If $W'$ denotes the preimage of $W$ under the  epimorphism $V\to V/V'$, one has a commutative diagram with exact rows
\[ \xymatrix{
0 \ar[r] & \iota(V') \ar[r] \ar[d]^{\cong}& N  \ar[r] \ar[d] & \iota(W) \ar[r] \ar[d]^{\cong}& 0\\
0 \ar[r] & \iota(V') \ar[r]  & \iota(W') \ar[r]   & \iota(W)  \ar[r] & 0
}, \]
 and therefore $N\cong \iota(W')$.
 
If $V\in \Vect_k$ has infinite dimension, we recall that $\iota(V)=\varinjlim\limits_{\substack{W\subset V\\ \text{fin.dim.}}} \iota(W)$ and hence, for any subobject $N\subseteq \iota(V)$, one has 
$$N=\varinjlim\limits_{\substack{W\subset V\\ \text{fin.dim.}}} N\cap \iota(W).$$
From the special case of finite dimension, we obtain $N\cap \iota(W)=\iota(W')$ for some $W'$ related to $W$, and therefore
$$N=\varinjlim\limits_{\substack{W\subset V\\ \text{fin.dim.}}} \iota(W')=\iota \left(\varinjlim\limits_{\substack{W\subset V\\ \text{fin.dim.}}} W'\right).$$
 
 Now let $V\in \Vect_k$ be arbitrary and, let $N$ be a quotient of $\iota(V)$. Then by the previous, $\Ker(\iota(V)\to N)$ is of the form $\iota(V')$ for some $V'\subseteq V$, and hence $N\cong \iota(V)/\iota(V')\cong \iota(V/V')$, as $\iota$ is exact.
\item By part (i), every sequence of subobjects $0=N_0\subsetneq N_1\subsetneq \dots \subsetneq \iota(V)$ is induced via $\iota$ by a sequence of subvector spaces $0=W_0\subsetneq W_1\subsetneq \dots  \subsetneq V$. Hence, ${\rm length}(\iota(V))=\dim_k(V)$.
\item  We use induction on the length of $M$. If $M$ has length $1$, then $M$ is a simple object. Since $\1$ also is simple, every morphism in $M^\C= \Mor_\C(\1,  M)$ is either $0$ or an isomorphism. In particular, $k=\End_\C(\1)$ acts transitively on $\Mor_\C(\1,  M)$, which shows that $\dim_k(\Mor_\C(\1,  M))$ is $0$ or $1$.
For the general case, take a subobject $0\ne N\ne M$ of $M$. Applying the functor $()^\C=\Mor_\C(\1, -)$ to the exact sequence $0\to N \to M \to M/N\to 0$ leads to an exact sequence
$$ 0 \to N^\C \to M^\C \to (M/N)^\C,$$
as the functor $ \Mor_\C(X,-)$ is always left-exact.\\
Hence, $\dim_k(M^\C)\leq \dim_k(N^\C)+\dim_k((M/N)^\C)$. Since $N$ and $M/N$ have smaller length than $M$, we obtain the claim by induction using ${\rm length}(M)={\rm length}(N)+{\rm length}(M/N)$.
\end{enumerate}
\end{proof}

\begin{prop}\label{prop:adjointness}
Let $\C$ be a category satisfying (C1) and (C2) and let $\iota$ and $()^\C$ be as above.
Then the following hold.
\begin{enumerate}
\item The functor $\iota$ is left adjoint to the functor $()^\C$, i.e.~for all $V\in \Vect_k$, $M\in \C$, there are isomorphisms of $k$-vector spaces
$\Mor_\C(\iota(V),M) \cong \Hom_k(V,M^\C)$ functorial in $V$ and $M$.
\item For every $V\in \Vect_k$, the homomorphism $\eta_V:V\to (\iota(V))^\C$ which is adjoint to $\id_{\iota(V)}$  is an isomorphism.
\item For every $M\in \C$, the morphism $\varepsilon_M:\1 \otimes_k \Mor_\C(\1, M)= \iota(M^\C)\to M$ which is adjoint to $\id_{M^\C}$ is a monomorphism.
\end{enumerate}
\end{prop}

\begin{rem}\label{rem:iota-full}
\begin{enumerate}
\item Whereas in the differential resp.~difference settings, part (i) and (ii) are easily seen, part (iii) amounts to saying that any set $v_1,\dots ,v_n\in M^\C$ of constant (resp.~invariant) elements of $M$ which are $k$-linearly independent, are also independent over the differential (resp.~difference) field $F$. This is proven in each setting separately. However, Amano and Masuoka provide an abstract proof (which is given in \cite[Prop.~3.1.1]{ka:ridepvt}) which relies on the Freyd embedding theorem.
\item The collection of homomorphisms $(\eta_V)_{V\in \Vect_k}$ is just the natural transformation $\eta:\id_{\Vect_k}\to (-)^\C\circ \iota$ (unit of the adjunction) whereas the morphisms $\varepsilon_M$ form the natural transformation $\varepsilon:\iota\circ (-)^\C\to \id_\C$ (counit of the adjunction).
By the general theory on adjoint functors, for all $V,W\in \Vect_k$, the maps $\Hom_k(V,W)\to \Mor_{\C}(\iota(V),\iota(W))$ induced by applying $\iota$ are just the compositions
$$\hspace*{15mm} \xymatrix@1@C+20pt{
\Hom_{k}(V,W) \ar[r]^(0.45){\eta_W\circ (-)} & \Hom_{k}(V,\iota(W)^{\C}) &\Mor_{\C}(\iota(V),\iota(W))
\ar[l]^{\simeq}_{\rm adjunction} }$$
(cf.~\cite{sml:cwm},p.~81,eq.~(3) and definition of $\eta$). This implies that $\eta_W$ is a monomorphism for all $W\in \Vect_k$ if and only if $\Hom_k(V,W)\to \Mor_{\C}(\iota(V),\iota(W))$ is injective for all $V,W\in \Vect_k$, i.e.~if $\iota$ is a faithful functor. Furthermore, $\eta_W$ is a split epimorphism for all $W\in \Vect_k$ if and only if $\Hom_k(V,W)\to \Mor_{\C}(\iota(V),\iota(W))$ is surjective for all $V,W\in \Vect_k$, if and only if $\iota$ is a full functor. In particular, $\eta_W$ being an isomorphism for all $W\in \Vect_k$ is equivalent to $\iota$ being a fully faithful functor.
%
\end{enumerate}
\end{rem}

\begin{proof}[Proof of Prop.~\ref{prop:adjointness}]
\begin{enumerate}[leftmargin=*, widest*=3]
\item For $V\in \vect_k$ and $M\in \C$ we have natural isomorphisms
\begin{eqnarray*}
\Mor_\C(\iota(V),M) &\cong & \Mor_\C(\1, M\otimes \iota(V)^\vee) \cong  \Mor_\C(\1,  M\otimes_k V^\vee)\\
&\cong &    \Mor_\C(\1,M)\otimes_k V^\vee\cong \Hom_k(V,  \Mor_\C(\1,M) ) \\
&= &   \Hom_k(V, M^\C)
\end{eqnarray*}
If $V$ is of infinite dimension the statement is obtained using that $\Mor_\C$ and $\Hom_k$ commute with inductive limits, i.e.
\begin{eqnarray*}
 \Mor_\C(\iota(V),M) &=& \Mor_\C(\varinjlim\limits_{\substack{W\subset V\\ \text{fin.dim}}} \iota(W), M) = \varprojlim\limits_{\substack{W\subset V\\ \text{fin.dim}}}  \Mor_\C(\iota(W),M)\\
&\cong &  \varprojlim\limits_{\substack{W\subset V\\ \text{fin.dim}}} \Hom_k(W,M^\C)
 = \Hom_k(V,M^\C).
\end{eqnarray*}
\item We have, $(\iota(V))^\C=\Mor_\C(\1,\1\otimes_k V)\cong  \Mor_\C(\1,\1)\otimes_k V\cong k\otimes_k V=V$, and the morphism $\id_{\iota(V)}$ corresponds to $\id_V:V\xrightarrow{\eta_V} (\iota(V))^\C\cong V$ via all these natural identifications.
\item Let $M\in \C$ and $N:=\Ker(\varepsilon_M)\subseteq  \iota(M^\C)$. By Prop.~\ref{prop:first-properties}(i), there is a subspace $W$ of $M^\C$ such that $N=\iota(W)$. Hence, we have an exact sequence of morphisms
$$0 \to \iota(W)\to \iota(M^\C) \xrightarrow{\varepsilon_M} M.$$
Since $()^\C$ is left exact, this leads to the exact sequence
$$0\to (\iota(W))^\C\to (\iota(M^\C))^\C \xrightarrow{(\varepsilon_M)^\C}\ M^\C$$
But $\eta_V:V\to (\iota(V))^\C$ is an isomorphism for all $V$ by part (ii). So we obtain an exact sequence
$$0\to W\to M^\C \xrightarrow{(\varepsilon_M)^\C \circ \eta_{M^\C}} M^\C,$$
and the composite $(\varepsilon_M)^\C \circ \eta_{M^\C}$ is the identity on $M^\C$ by general theory on adjoint functors (cf.~\cite[Ch.~IV, Thm.~1]{sml:cwm}). Hence, $W=0$.

\end{enumerate}
\end{proof}

\section{$\C$-algebras and base change}\label{sec:c-algebras}

We recall some notation which are already present in \cite[Ch.~17 \& 18]{sml:ca}, and refer to loc.~cit.~for more details. The reader should be aware that a ``tensored category'' in \cite{sml:ca} is the same as an ``abelian symmetric monoidal category''  here.\\ 
A {\markdef commutative algebra in $\C$} (or a {\markdef $\C$-algebra} for short) is an object $R\in \C$ together with two morphisms $u_R:\1\to R$ and $\mu_R:R\otimes R\to R$ satisfying several commuting diagrams corresponding to associativity, commutativity and the unit. For instance,
$$\mu_R\circ (u_R\otimes \id_R)=\id_R=\mu_R\circ (\id_R\otimes u_R)$$
says that $u_R$ is a unit for the multiplication $\mu_R$ (cf.~\cite[Ch.~17]{sml:ca}; although the term ``$\C$-algebra'' in \cite{sml:ca} does not include commutativity).

For a $\C$-algebra $R$ we define $\C_R$ to be the category of $R$-modules in $\C$, i.e. the category of pairs $(M,\mu_M)$ where $M\in \C$ and $\mu_M:R\otimes M\to M$ is a morphism in $\C$ satisfying the usual commuting diagrams for turning $M$ into an $R$-module (cf.~\cite[Ch.~18]{sml:ca}).\footnote{Most times, we will omit the $\mu_M$ in our notation, and just write $M\in \C_R$.} The morphisms in $\C_R$ are morphisms in $\C$ which commute with the $R$-action. The category $\C_R$ is also an abelian symmetric monoidal category with tensor product $\otimes_R$ defined as
$$M\otimes_R N:={\rm Coker}((\mu_M\circ \tau)\otimes \id_N- \id_M\otimes \mu_N:M\otimes R\otimes N\to M\otimes N),$$
where $\tau:M\otimes R\to R\otimes M$ is the twist morphism (see~\cite[Prop.~18.3]{sml:ca}).

A {\markdef $\C$-ideal} $I$ of a $\C$-algebra $R$ is a subobject of $R$ in the category $\C_R$, and $R$ is called a {\markdef simple} $\C$-algebra, if $0$ and $R$ are the only $\C$-ideals of $R$, i.e.~if $R$ is a simple object in $\C_R$.

\begin{defn}
For a $\C$-algebra $R$, the additive right-exact functor $()_R:(\C,\otimes)\to (\C_R,\otimes_R), M\mapsto M_R:=(R\otimes M, \mu_R\otimes \id_M)$ is called the \emph{base change functor}. It is even a tensor functor, and it is a left adjoint to the forgetful functor $\C_R\to \C$ (see~\cite[Thm.~18.2]{sml:ca}). \\
We can also base change the functors $\iota$ and $()^\C$. In more details, having in mind that $\End_{\C_R}(R)=\Mor_\C(\1,R)=R^\C$:
$$\iota_R:\mod_{R^\C}\to \C_R, V\mapsto R\otimes_{\iota(R^\C)} \iota(V)$$ 
and
$$()^{\C_R}:\C_R\to \mod_{R^\C}, M\mapsto \Mor_{\C_R}(R,M)=\Mor_{\C}(\1,M)=M^\C.$$
 
\medskip

A special case is given, if $R=\iota(A)$ for some commutative $k$-algebra $A$. In this case, $\iota_R$ is ``the same'' as $\iota$. This case corresponds to an extension by constants in the theory of differential or difference modules.
\end{defn}

\begin{prop}
The functor $\iota_R$ is left adjoint to the functor $()^{\C_R}$. 
\end{prop}

\begin{proof}
Let $V\in \mod_{R^\C}$ and $M\in \C_R$, then
$$\Mor_{\C_R}(\iota_R(V),M)=\Mor_{\C_R}(R\otimes_{\iota(R^\C)} \iota(V),M)=\Mor_{\C_{\iota(R^\C)}}(\iota(V),M)$$
is the subset of $\Mor_{\C}(\iota(V),M)$ given by all $f\in \Mor_{\C}(\iota(V),M)$ such that the diagram
$$\xymatrix{
\iota(R^\C)\otimes \iota(V) \ar[r]^{\id\otimes f} \ar[d]^{\iota(\mu_V)} & \iota(R^\C)\otimes M \ar[d]^{\mu_M} \\
\iota(V)  \ar[r]^{f} & M
}$$
commutes. On the other hand, $\Hom_{R^\C}(V,M^{\C_R})=\Hom_{R^\C}(V,M^\C)$ is the subset of $\Hom_{k}(V,M^\C)$  given by all $g\in \Hom_{k}(V,M^\C)$ such that the diagram
$$\xymatrix{
R^\C\otimes_k V \ar[r]^{\id\otimes g} \ar[d]^{\mu_V} & R^\C\otimes_k M^\C \ar[d]^{(\mu_M)^\C} \\
V  \ar[r]^{g} & M^\C
}$$
commutes. Assume that $f$ and $g$ are adjoint morphisms (i.e.~correspond to each other via the bijection $\Mor_{\C}(\iota(V),M)\cong \Hom_{k}(V,M^\C)$ of Prop.\ref{prop:adjointness}(i)), then the commutativity of the first diagram is equivalent to the commutativity of the second, since the bijection of the hom-sets is natural.
\end{proof}

\begin{lemma}\label{lemma:ideal-bijection::abstract}
Let $A$ be a commutative $k$-algebra. Then $\iota_{\iota(A)}$ and $()^{\C_{\iota(A)}}$ define a bijection between the ideals of $A$ and the $\C$-ideals of $\iota(A)$.
\end{lemma}

\begin{proof}
By definition $\iota_{\iota(A)}(I)=\iota(I)$ for any $I\in \mod_A$. Furthermore, by Prop.~\ref{prop:first-properties}(i), $\iota$ induces a bijection between the $k$-subvector spaces of $A$ and the subobjects of $\iota(A)$ in $\C$. The condition on $I$ being an ideal of $A$  (resp.~of $\iota(I)$ being an ideal of $\iota(A)$) is equivalent to the condition that the composite $A\otimes_k I\xrightarrow{\mu_A} A\to A/I$ (resp.~the composite $\iota(A)\otimes \iota(I)\xrightarrow{\mu_{\iota(A)}} \iota(A)\to \iota(A)/\iota(I)$) is the zero map. Hence, the condition for $\iota(I)$ is obtained from the one for $I$ by applying $\iota$, and using that $\iota$ is an exact tensor functor. Since $\iota$ is also faithful, these two conditions are indeed equivalent.
\end{proof}

In the special case that $A$ is a field, one obtains the following corollary.

\begin{cor}\label{cor:still-simple::abstract}
Let $\ell$ be a field extension of $k$, then $\iota(\ell)$ is a simple $\C$-algebra.
\end{cor}

\begin{rem}
As $\iota_R$ and $()^{\C_R}$ are adjoint functors, there are again the unit and the counit of  the adjunction.
By abuse of notation, we will again denote the unit by $\eta$ and the counit by $\eps$. There might be an ambiguity which morphism is meant by $\eps_M$ if $(M,\mu_M)$ is an object in $\C_R$. However, when $M$ is explicitly given as an object of $\C_R$, then $\eps_M:\iota_R(M^{\C_R})\to M$ is meant. This is the case, for example, if $M=N_R$ is the base change of an object $N\in \C$.\\
In cases where the right meaning of $\eps_M$  would not be clear, we always give the source and the target of $\eps_M$.
\end{rem}

\begin{prop}\label{prop:on-iota-r}
Assume that, $\iota_R$ is exact and faithful \footnote{For differential rings this means that the ring $R$ is faithfully flat over $\iota(R^\C)$.}, and that any subobject of $R^n$ is of the form $\iota_R(W)$, then  the following holds.
\begin{enumerate}
\item For every $V\in \Mod_{R^\C}$, every subobject of $\iota_R(V)$ is isomorphic to $\iota_R(W)$ for some $W\subseteq V$.
\item For every $V\in \Mod_{R^\C}$, the morphism $\eta_V:V\to (\iota_{R}(V))^{\C_R}$ is an isomorphism.
\item For every $M\in \C_R$, the morphism $\varepsilon_M:\iota_R(M^{\C_R})\to M$ is a monomorphism.
\end{enumerate}
\end{prop}

The most important cases where the proposition applies is on the one hand the case $R=\iota(A)$ for some commutative $k$-algebra $A$ (in which case $\iota_R=\iota$), and on the other hand $R$ being a simple $\C$-algebra.

\begin{proof}
\begin{enumerate}[leftmargin=*, widest*=3] \setcounter{enumi}{1}
\item We show that $\eta_V:V\to (\iota_{R}(V))^{\C_R}$ is an isomorphism for all $V\in  \Mod_{R^\C}$. 
As $\iota$ is faithful by assumption, all $\eta_V$ are monomorphisms (cf.~Rem.~\ref{rem:iota-full}).
For showing that $\eta_V$ is an epimorphism, it is enough to show that the natural map
$$R^\C\otimes_k V=(R\otimes \iota(V))^{\C_R}\to (\iota_R(V))^{\C_R}$$
is an epimorphism, where on the left hand side, $V$ is considered just as a $k$-vector space. Saying that  this map is epimorphic is equivalent to saying that any morphism $g:R\to \iota_R(V)$ in $\C_R$ can be lifted to 
a morphism $f:R\to R\otimes \iota(V)$ in $\C_R$. So let  $g:R\to \iota_R(V)$ be a morphism in $\C_R$, and let $P$ be the pullback of the diagram
$$\xymatrix{
P \ar@{->>}[r]^{{\rm pr}_1} \ar[d]^{{\rm pr}_2} & R\ar[d]^{g} \\
R\otimes \iota(V)  \ar@{->>}[r]^(.35){p} & \iota_R(V)   
}.$$
Then $P$ is a subobject of $R\oplus (R\otimes \iota(V))\cong R^{1+\dim_k(V)}$, and hence by assumption, $P=\iota_R(W)$ for some $W\in \mod_{R^\C}$. By adjointness, ${\rm pr}_1$ corresponds to some $R^\C$-homomorphism $q:W\to R^{\C_R}=R^\C$, i.e.~${\rm pr}_1=\varepsilon_R\circ \iota_R(q)$. Since $\varepsilon_R:R=\iota_R\left(R^{\C_R}\right)\to R$ is the identity, and ${\rm pr}_1$ is an epimorphism, faithfulness of $\iota_R$ implies that also $q$ is an epimorphism. Therefore, there is a $R^\C$-homomorphism $s:R^\C\to W$ such that $q\circ s=\id_{R^\C}$.
Let $f$ be the morphism $f:={\rm pr}_2\circ \iota_R(s):R\to R\otimes \iota(V)$, then
$$p\circ f=p\circ {\rm pr}_2\circ \iota_R(s)
= g\circ {\rm pr}_1\circ \iota_R(s)=g\circ  \iota_R(q\circ s)=g.$$
Hence, $f$ is a lift of $g$.

\setcounter{enumi}{0}
\item We show that any subobject of $\iota_R(V)$ is of the form $\iota_R(W)$ for some submodule $W$ of $V$. The case of a quotient of $\iota_R(V)$ then follows in the same manner as in Prop.~\ref{prop:first-properties}.
Let $N\subseteq \iota_R(V)$ be a subobject in $\C_R$. Then the pullback of $N$ along $p:R\otimes \iota(V)\to \iota_R(V)$ is a subobject of $R\otimes \iota(V)$, hence by assumption of the form $\iota_R(\tilde{W})$ for some $\tilde{W}\subseteq (R^\C)^{\dim_k(V)}$.
Furthermore, as $\eta_V$ is an isomorphism, the restriction $p|_{\iota_R(\tilde{W})}: \iota_R(\tilde{W})\to \iota_R(V)$ is induced by some homomorphism $f:\tilde{W}\to V$ (cf. Remark \ref{rem:iota-full}). By exactness of $\iota_R$, we finally obtain $N=\im(\iota_R(f))=\iota_R(\im(f))=\iota_R(W)$ for $W:= \im(f)$.
\setcounter{enumi}{2}
\item The proof that $\varepsilon_M:\iota_R(M^{\C_R})\to M$ is a monomorphism is the same as in Prop.~\ref{prop:adjointness}.
\end{enumerate}
\end{proof}

\begin{lemma}\label{lemma:when-eps-is-iso}
Let $R$ be a simple $\C$-algebra. Then for $N\in \C_R$, the morphism $\eps_N$ is an isomorphism if and only if $N$ is isomorphic to $\iota_R(V)$ for some $V\in \Mod_{R_\C}$.
\end{lemma}

\begin{proof}
If $\eps_N$ is an isomorphism, then $N\isom \iota_R(V)$ for $V:=N^{\C_R}$. On the other hand, let $N\isom \iota_R(V)$ for some
$V\in \Mod_{R_\C}$. Since $\iota_R(\eta_V)\circ \eps_{\iota_R(V)}=\id_{\iota_R(V)}$ (cf.~\cite[Ch.~IV, Thm.~1]{sml:cwm}) and $\eta_V$ is an isomorphism, $\eps_{\iota_R(V)}$ is an isomorphism. Hence, $\eps_N$ is an isomorphism.
\end{proof}

\begin{prop}\label{prop:subcat-of-trivial-modules}
Let $R$ be a simple $\C$-algebra. Then the full subcategory of $\C_R$ consisting of all $N\in \C_R$ such that $\eps_N$ is an isomorphism is a monoidal subcategory of $\C_R$ and is closed under taking direct sums, subquotients, small inductive limits, and duals of dualizable objects in $\C_R$.
\end{prop}

\begin{proof}
Using the previous lemma, this follows directly from Prop.~\ref{prop:on-iota-r}(i), and the fact that $\iota_R$ is an additive exact tensor functor.
\end{proof}

\section{Solution rings and Picard-Vessiot rings}\label{sec:solution-rings}

From now on we assume that $\C$ satisfies all conditions (C1), (C2), (F1) and (F2).

\begin{lemma}\label{lemma:dualizables-are-projective-of-finite-rank}
Let $M\in \C$ be dualizable. Then $\upsilon(M)$ is a finitely generated locally free $\O_{\X}$-module of constant rank.
\end{lemma}

\begin{proof}
If $M\in \C$ is dualizable, then $\upsilon(M)$ is dualizable in $\Qcoh(\X)$, since  $\upsilon$ is a tensor functor, and tensor functors map dualizable objects to dualizable objects (and their duals to the duals of the images).  By \cite[Prop.~2.6]{pd:ct}, dualizable objects in $\Qcoh(\X)$ are exactly the finitely generated locally free $\O_\X$-modules. Hence, $\upsilon(M)$ is finitely generated and locally free whenever $M$ is dualizable.

To see that the rank is constant, let $d\in\NN$ be the maximal local rank of $\upsilon(M)$, and consider the $d$-th exterior power $\Lambda:=\Lambda^d(M)\in \C$ which is non-zero by the choice of $d$. Hence, the evaluation morphism $\ev_{\Lambda}:\Lambda\otimes \Lambda^\vee\to \1$ is non-zero. Since $\1$ is simple, and the image of $ \ev_{\Lambda}$ is a subobject of $\1$, the morphism $\ev_{\Lambda}$ is indeed an epimorphism. Hence the evaluation
$$\ev_{\upsilon(\Lambda)}=\upsilon( \ev_{\Lambda}):\upsilon(\Lambda)\otimes_{\O_\X} \upsilon(\Lambda)^\vee\to \O_\X$$
is surjective which implies that $\upsilon(\Lambda)\otimes_{\O_\X} \O_{\X,x}\ne 0$ for any 
point $x$ of $\X$. But this means that any local rank of $\upsilon(M)$ is at least $d$, i.e.~$\upsilon(M)$ has constant rank $d$.
\end{proof}

\begin{rem}\label{rem:dualizables-are-projective}
With respect to the previous lemma, condition (F2) implies that if $\upsilon(M)$ is finitely generated for some $M\in\C$, then $\upsilon(M)$ is even locally free and of constant rank. This also implies the following:\\
If $M$ is dualizable, then $\upsilon(M)$ is finitely generated and locally free. Further, for
every epimorphic image $N$ of $M$, the ${\O_\X}$-module $\upsilon(N)$ is also finitely generated and hence, locally free. But then for any subobject $N'\subseteq M$ the sequence $0\to \upsilon(N')\to \upsilon(M)\to \upsilon(M/N')\to 0$ is split exact, since $\upsilon(M/N')$ as an epimorphic image is locally free.
Therefore $\upsilon(N')$ is also a quotient of $\upsilon(M)$, in particular $\upsilon(N')$ is finitely generated and locally free.\\
So given a dualizable $M\in \C$, all subquotients of finite direct sums of objects $M^{\otimes n}\otimes (M^\vee)^{\otimes m}$ ($n,m\in \NN$) are dualizable. Hence, the strictly full tensor subcategory of $\C$ generated by $M$ and $M^\vee$ -- which is exactly the full subcategory of $\C$ consisting of all objects isomorphic to subquotients of finite direct sums of objects $M^{\otimes n}\otimes (M^\vee)^{\otimes m}$ ($n,m\in \NN$) -- 
 is a rigid abelian tensor category and will be denoted by $\tenscat{M}$. Furthermore by definition, $\upsilon$ is a fibre functor and therefore $\tenscat{M}$ is even a Tannakian category (cf.~\cite[Section 2.8]{pd:ct}).

By \cite[Cor.~6.20]{pd:ct}, there exists a finite extension $\tilde{k}$ of $k$ and a fibre functor $\omega:\tenscat{M}\to \vect_{\tilde{k}}$. In view of Thm.~\ref{thm:pv-rings-equiv-to-fibre-functors} in Section \ref{sec:pv-rings-and-fibre-functors}, this implies that there is a Picard-Vessiot ring for $M$ over $\tilde{k}$.

We will see later (cf.~Cor.~\ref{cor:properties-of-simple-minimal-solution-rings}) that for every simple minimal solution ring $R$, the field $R^\C=\End_{\C_R}(R)$ is a finite field extension of $k$.
\end{rem}

\begin{defn}
Let $M\in \C$.\\ 
A \emph{solution ring} for $M$ is a $\C$-algebra $R$ such that the morphism
$$\varepsilon_{M_R}: \iota_R\left( (M_R)^{\C_R} \right) \to M_R=R\otimes M$$
is an isomorphism.

A \emph{Picard-Vessiot ring} for $M$ is a minimal solution ring $R$ which is a simple $\C$-algebra, and satisfies $ R^\C:=\End_{\C_R}(R)=k$. Here, \emph{minimal} means that for any solution ring $\tilde{R}\in \C$ that admits a monomorphism of $\C$-algebras to $R$, this monomorphism is indeed an isomorphism.
\end{defn}

\begin{rem}
Comparing with the differential setting, $(M_R)^{\C_R}$ is just the so called solution space $(R\otimes_F M)^\partial$ of $M$ over $R$, and $\eps_{M_R}$ is the canonical homomorphism 
$R\otimes_{R^\partial} (R\otimes_F M)^\partial \to R\otimes_F M$.\\
When $R$ is a simple $\C$-algebra (i.e.~in the differential setting a simple differential ring), then by Prop.\ref{prop:on-iota-r}(iii), $\eps_{M_R}$ is always a monomorphism. Hence, for a simple $\C$-algebra $R$, the condition for being a solution ring means that the solution space is as large as possible, or in other words that $R\otimes M$ has a basis of constant elements, i.e.~is a trivial differential module over $R$.
\end{rem}

\begin{prop}\label{prop:image-of-solution-rings}
Let $R$ be a solution ring for some dualizable $M\in \C$, and let $f:R\to R'$ be an epimorphism of $\C$-algebras.
Assume either that $R'$ is a simple $\C$-algebra or that $(R\otimes M)^\C$ is a free $R^\C$-module. Then $R'$ is a solution ring for $M$ as well. 
\end{prop}

\begin{rem}
 If $(R\otimes M)^\C$ is a free $R^\C$-module, then it is automatically free of finite rank, and the rank is the same as the global rank of $\upsilon(M)$ as
$\O_\X$-module which exists by Lem.~\ref{lemma:dualizables-are-projective-of-finite-rank}.
\end{rem}

\begin{proof}[Proof of Prop.~\ref{prop:image-of-solution-rings}]
 As $f:R\to R'$ is an epimorphism and $M$ is dualizable, $f\otimes \id_M:R\otimes M\to R'\otimes M$ is an epimorphism, too.
As the diagram
$$\xymatrix@C+20pt{
 \iota_R\left( (R\otimes M)^\C\right) \ar[r]^{\eps_{M_R}} \ar[d] &
M_R=R\otimes M \ar@{->>}[d]^{f\otimes \id_M} \\
\iota_{R'}\left( (R'\otimes M)^\C\right) \ar[r]^{\eps_{M_{R'}}}  &
M_{R'}=R'\otimes M
}$$
commutes and $\eps_{M_R}$ is an isomorphism by assumption on $R$, the morphism $\eps_{M_{R'}}$ is an epimorphism.

If $R'$ is simple, then by Prop.~\ref{prop:on-iota-r}(iii) the morphism $\eps_{M_{R'}}$ is a monomorphism, hence an isomorphism.
Therefore, $R'$ is a solution ring.

Assume now, that $(R\otimes M)^\C$ is a free $R^\C$-module of rank $n$. Then $\iota_R(R\otimes M)^\C\cong \iota_R\left((R^\C)^n\right)=R^n$.
Composing with $\eps_{M_R}$ leads to an isomorphism $R^n\xrightarrow{\cong}R\otimes M$. We therefore obtain an isomorphism $\alpha:(R')^n\to R'\otimes M$ by
tensoring with $R'$.
Applying the natural transformation $\eps$ to this isomorphism, we get a commutative square
$$\xymatrix@C+20pt{
R'^n=\iota_{R'}\bigl( (R'^n)^\C\bigr) \ar[r]^{\iota_{R'}(\alpha^\C)}_{\cong} \ar[d]^{\eps_{R'^n}}_{\cong} &
\iota_{R'}\bigl( (R'\negotimes M)^\C\bigr) \ar[d]^{\eps_{M_{R'}}} \\
R'^n \ar[r]^{\alpha}_{\cong} & R'\negotimes M,
}$$
which shows that $\eps_{M_{R'}}$ is an isomorphism, too.
\end{proof}

%
%

\begin{thm}\label{thm:exists-sol-ring}
Let $M\in \C$ be dualizable. Then there exists a non-zero solution ring for~$M$. 
\end{thm}

\begin{proof}
We show the theorem by explicitly constructing a solution ring. This construction is motivated by the Tannakian point of view in \cite{pd-jsm:tc} and by Section 3.4 in   \cite{ya:dnctgdd}.\\
Let $n:={\rm rank}(\upsilon(M))$ be the global rank of the $\O_\X$-module $\upsilon(M)$ which exists by Lemma \ref{lemma:dualizables-are-projective-of-finite-rank}. We then define $U$ to be the residue ring of $\Sym\Bigl( \left(M\otimes (\1^n)^\vee\right) \oplus \left(\1^n\otimes M^\vee\right) \Bigr)$ subject to the ideal generated by the image of the morphism
\begin{eqnarray*}
(-\ev, \id_{M}\otimes \delta_{\1^n}\otimes \id_{M^\vee}):M\otimes M^\vee &\to & \1\oplus \left(  M\otimes  (\1^n)^\vee \otimes \1^n \otimes M^\vee\right)\\ &\subset& \Sym\Bigl( \left(M\otimes (\1^n)^\vee\right) \oplus \left(\1^n\otimes M^\vee\right) \Bigr) .
\end{eqnarray*}

First we show that $U\ne 0$ by showing $\upsilon(U)\ne 0$. By exactness of $\upsilon$, the ring $\upsilon(U)$ is given as the residue ring of $\Sym\left( (\upsilon(M) \otimes_{\O_\X} (\O_{\!\X}^{\,n})^\vee) \oplus
(\O_{\!\X}^{\,n}\otimes_{\O_\X} \upsilon(M)^\vee)\right)$ subject to the ideal generated by the image of $(-\ev_{\upsilon(M)}, \id\otimes \delta_{\O_{\!\X}^{\,n}}\otimes \id)$.

Let $\U=\spec(S)\subseteq \X$ be an affine open subset such that $\tilde{M}:=\upsilon(M)(\U)$ is free over $S$. Let $\{b_1,\ldots,b_n\}$ be a basis of $\tilde{M}$ and $b_1^\vee,\ldots, b_n^\vee\in \tilde{M}^\vee$ the corresponding dual basis. 
Then $\upsilon(U)(\U)$ is generated by $x_{ij}:=b_i\otimes e_j^\vee\in \tilde{M} \otimes (S^{n})^\vee$ and $y_{ji}:=e_j\otimes b_i^\vee\in S^{n}\otimes (\tilde{M})^\vee$ for $i,j=1,\ldots, n$, where $\{e_1,\ldots, e_n\}$ denotes the standard basis of $S^n$ and $\{e_1^\vee,\ldots, e_n^\vee\}$ the dual basis.
The relations are generated by
$$b_k^\vee(b_i)=\ev_{\tilde{M}}(b_i\otimes b_k^\vee)=
(\id_{\tilde{M}}\otimes \delta_{S^n}\otimes \id_{\tilde{M}^\vee})(b_i\otimes b_k^\vee)=
\sum_{j=1}^n (b_i\otimes e_j^\vee)\otimes (e_j\otimes b_k^\vee),$$ i.e.~$\delta_{ik}=\sum_{j=1}^n x_{ij}y_{jk}$ for all $i,k=1,\ldots, n$. This just means that the matrix $Y=(y_{jk})$ is the inverse of the matrix $X=(x_{ij})$. Hence $\upsilon(U)(\U)=S[X,X^{-1}]$ is the localisation of a polynomial ring over $S$ in $n^2$ variables.

For showing that $U$ is indeed a solution ring, we consider the following diagram
$$\xymatrix@R+10pt@C+30pt{
M  \ar[r]^{ \id_M\otimes \delta_{\1^n}} \ar[d]^{ \id_M\otimes \delta_M} & 
 (M \negotimes (\1^n)^\negvee)\negotimes \1^n \ar[r]^{ \text{incl.}\otimes \id_{\1^n}}  \ar[d]^{ \id\otimes \delta_M}  & 
 U\negotimes \1^n \ar[d]^{ \id\otimes \delta_M} \\
M\negotimes M^\negvee \negotimes M \ar[r]^(.4){\id_M\otimes \delta_{\1^n}\otimes \id}  \ar[d]^{ \ev_M\otimes \id_M} & 
(M\negotimes (\1^n)^\negvee)\negotimes (\1^n \negotimes M^\negvee)\negotimes M \ar[r]^(.55){\text{incl.}\otimes \id} \ar[d]^{ \mu_U\otimes \id_M}&
U\negotimes (\1^n\negotimes M^\negvee) \negotimes M \ar[d]^{\mu_U\otimes \id_M} \\
\1\negotimes M \ar[r]^{u_U\otimes \id_M} & 
U\negotimes M \ar[r]^{\id}& 
U\negotimes M.
}$$
It is easy to see that the upper left, upper right and lower right squares all commute. The lower left square also commutes by definition of $U$, since the difference of the two compositions in question is just $(-\ev_M,\id_{M}\otimes \delta_{\1^n}\otimes \id_{M^\vee})\otimes \id_M$.
Furthermore the composition of the two vertical arrows on the left is just the identity on $M$ by definition of the dual. Tensoring the big square with $U$ leads to the left square of the next diagram
$$\xymatrix@C+30pt{
U\otimes M \ar[r] \ar[d]^{\id} &
U\otimes U\otimes \1^n \ar[r]^{ \mu_U\otimes \id_{\1^n}} \ar[d]^{\id_U \otimes\alpha} &
U\otimes\1^n \ar[d]^{\alpha}\\
U\otimes M  \ar[r]^(.45){\id_U\otimes u_U\otimes \id_M}&
U\otimes U \otimes M \ar[r]^{ \mu_U\otimes\id_M}& 
U\otimes M\\
}$$
where $\alpha:=( \mu_U\otimes\id_M)\circ (\id\otimes \delta_M)$. The right square of this diagram also commutes, as is easily checked, and the composition in the bottom row is just the identity according to the constraints on the unit morphism $u_U$ and the multiplication map $\mu_U$. Hence, $\alpha:U\otimes \1^n\to U\otimes M$ is a split epimorphism in $\C$, and even in $\C_U$ (since the right square commutes). Since the rank of $\upsilon(U\otimes \1^n)=\upsilon(U)^n$ and the rank of $\upsilon(U\otimes M)$ as $\upsilon(U)$-modules are both $n$, the split epimorphism $\upsilon(\alpha)$ is in fact an isomorphism, i.e.~$\alpha$ is an isomorphism.

Applying the natural transformation $\eps$, we finally obtain the commutative square
$$\xymatrix@C+20pt{
U^n=\iota_U\bigl( (U \negotimes \1^n)^\C\bigr) \ar[r]^{\iota_U(\alpha^\C)}_{\cong} \ar[d]^{\eps_{U^n}}_{\cong} &
\iota_U\bigl( (U\negotimes M)^\C\bigr) \ar[d]^{\eps_{M_U}} \\
U^n=U\negotimes \1^n \ar[r]^{\alpha}_{\cong} & U\negotimes M,
}$$
which shows that $\eps_{M_U}$ is an isomorphism. Hence, $U$ is a solution ring for $M$.
\end{proof}

\begin{rem}\label{rem:universal-solution-ring}
In the case of difference or differential modules over a difference or differential field $F$, respectively, the ring $U$ constructed in the previous proof is just the usual universal solution algebra $F[X,\det(X)^{-1}]$ for a fundamental solution matrix $X$ having indeterminates as entries. We will therefore call $U$ the {\markdef universal solution ring} for $M$.\\
This is moreover justified by the following theorem which states that $U$ indeed satisfies a universal property.
\end{rem}

\begin{thm}\label{thm:univ-sol-ring}
Let 
 $R$ be a solution ring for $M$, such that $(R\otimes M)^\C$ is a free $R^\C$-module, and let $U$ be the solution ring for $M$ constructed  in Thm.~\ref{thm:exists-sol-ring}. Then there exists a morphism of $\C$-algebras $f:U\to R$. Furthermore, the image of $\iota(R^\C)\otimes U\xrightarrow{\eps_R\otimes f} R\otimes R\xrightarrow{\mu_R} R$  does not depend on the choice of $f$.
\end{thm}

\begin{proof}
By assumption, we have an isomorphism in $\C_R$:
$$\alpha: R^n\xrightarrow{\cong} \iota_R\left( (M_R)^{\C_R} \right)=R\otimes_{\iota(R^\C)} \iota\left((R\otimes M)^\C\right) \xrightarrow{\cong} R\otimes M.$$
Since $M$ is dualizable, one has bijections
\begin{align*}
\Mor_{\C_R}(R^n,R\otimes M) &\simeq \Mor_{\C_R}(R\otimes (\1^n\otimes M^\vee), R) &\simeq \Mor_\C(\1^n\otimes M^\vee,R) \\
\alpha &\mapsto  \tilde{\alpha}_R:=(\id_R\otimes \ev_M)\circ (\alpha\otimes \id_{M^\vee}) &\mapsto  \tilde{\alpha}:=\tilde{\alpha}_R|_{\1^n\otimes M^\vee}
\end{align*}
Similarly, for the inverse morphism $\beta:=\alpha^{-1}:R\otimes M\to R^n$, one has
\begin{align*}
\Mor_{\C_R}(R\otimes M,R^n) &\simeq \Mor_{\C_R}(R\otimes (M\otimes (\1^n)^\vee), R) &\simeq \Mor_\C(M\otimes (\1^n)^\vee,R) \\
\beta &\mapsto  \tilde{\beta}_R:=(\id_R\otimes \ev_{\1^n})\circ (\beta\otimes \id_{(\1^n)^\vee}) &\mapsto  \tilde{\beta}:=\tilde{\beta}_R|_{M\otimes (\1^n)^\vee}
\end{align*}
Therefore the isomorphism $\alpha$ induces a morphism of $\C$-algebras
$$f:\Sym\Bigl( (M\otimes (\1^n)^\vee )\oplus (\1^n\otimes M^\vee) \Bigr)\to R.$$
We check that this morphism factors through $U$, i.e.~we have to check that
the morphisms
$$M\otimes M^\vee \xrightarrow{\id\otimes\delta_{\1^n}\otimes \id} M \otimes (\1^n)^\vee\otimes \1^n \otimes M^\vee \xrightarrow{ \tilde{\beta} \otimes \tilde{\alpha}}R\otimes R\xrightarrow{\mu_R} R $$
and
$$M\otimes M^\vee \xrightarrow{\ev_M} \1 \xrightarrow{u_R} R$$
are equal.
For this we consider the $R$-linear extensions in the category $\C_R$. By \cite[Sect.~2.4]{pd:ct}, the composition
$$M_{\!R}^\vee\xrightarrow{\delta_{R^n}\otimes \id_{M_{\!R}^\vee}} (R^n)^\vee \otimesR R^n\otimesR M_{\!R}^\vee  \xrightarrow{\id\otimes \alpha\otimes \id} M_{\!R}^\vee\otimesR M_{\!R}\otimesR (R^n)^\vee \xrightarrow{ \id\otimes \ev_{M_{\!R}}} (R^n)^\vee$$
is just the transpose $\tp{\alpha}:M_R^\vee\to (R^n)^\vee$ of the morphism $\alpha$, and this equals the contragredient $\beta^\vee$ of $\beta=\alpha^{-1}$.\\
Hence the equality of the two morphisms reduces to the commutativity of the 
diagram
$$\xymatrix@C+20pt{
M_R \otimes_R M_R^\vee \ar[r]^{\beta \otimes \beta^\vee} \ar[dr]_{\ev_{M_R}} &
R^n \otimes_R (R^n)^\vee   \ar[d]^{\ev_{R^n}} \\
& R.
}$$
But by definition of the contragredient (see \cite[Sect.~2.4]{pd:ct}), this diagram commutes.

It remains to show that the image of $\iota(R^\C)\otimes U\xrightarrow{\eps_R\otimes f} R\otimes R\xrightarrow{\mu_R} R$ does not depend on the chosen morphism $f:U\to R$.\\
Given two morphism of $\C$-algebras $f,g:U\to R$, let $\tilde{\alpha}_f, \tilde{\alpha}_g\in \Mor_{\C}(\1^n \otimes M^\vee, R)$ be the restrictions of $f$ resp. of $g$ to $\1^n\otimes M^\vee\subseteq U$, and let $\tilde{\beta}_f, \tilde{\beta}_g\in \Mor_{\C}(M\otimes (\1^n)^\vee,R)$ be the restrictions of $f$ resp. of $g$ to $M\otimes (\1^n)^\vee\subseteq U$.
Furthermore, let $\alpha_f,\alpha_g\in \Mor_{\C_R}(R^n,M_R)$ and $\beta_f,\beta_g\in \Mor_{\C_R}(M_R,R^n)$ denote the corresponding isomorphisms. Then by similar considerations as above one obtains that $\beta_f$ and $\beta_g$ are the inverses of $\alpha_f$ and $\alpha_g$, respectively.
Then
$$\beta_g\circ \alpha_f\in \Mor_{\C_R}(R^n,R^n)\simeq \Hom_{R^\C}((R^\C)^n,(R^n)^\C)
\simeq  \Mor_{\C_{\iota(R^\C)}}(\iota(R^\C)^n,\iota(R^\C)^n)$$
is induced by an isomorphism on $\iota(R^\C)^n$ (which we also denote by $\beta_g\circ \alpha_f$).
Therefore for the $\iota(R^\C)$-linear extension $\tilde{\alpha}_{f,\iota(R^\C)},\tilde{\alpha}_{g,\iota(R^\C)}:\iota(R^\C)\otimes \1^n\otimes M^\vee \to R$, one has
\begin{eqnarray*}
\tilde{\alpha}_{f,\iota(R^\C)} &=& (\id_R\otimes \ev_M)\circ (\alpha_f|_{\iota(R^\C)^n}\otimes \id_{M^\vee}) \\
&=& (\id_R\otimes \ev_M)\circ (\alpha_g|_{\iota(R^\C)^n}\otimes \id_{M^\vee})\circ \left((\beta_g\circ \alpha_f)\otimes  \id_{M^\vee}\right) \\
&=& \tilde{\alpha}_{g,\iota(R^\C)}\circ  \left((\beta_g\circ \alpha_f)\otimes  \id_{M^\vee}\right).
\end{eqnarray*}
and similarly,
$$\tilde{\beta}_{f,\iota(R^\C)}=  \tilde{\beta}_{g,\iota(R^\C)}\circ  \left((\alpha_g\circ \beta_f)\otimes  \id_{M^\vee}\right).$$
Hence, the morphism $\mu_R\circ (\eps_R\otimes f):\iota(R^\C)\otimes U\to R$ factors through $\mu_R\circ (\eps_R\otimes g)$ and by changing the roles of $f$ and $g$, the morphism $\mu_R\circ (\eps_R\otimes g)$ factors through $\mu_R\circ (\eps_R\otimes f)$. So the images are equal.
\end{proof}

\begin{rem}
In the classical settings, every Picard-Vessiot ring for some module $M$ is a quotient of the universal solution ring $U$. This is also the case in this abstract setting (see Thm.~\ref{thm:simple-minimal-solution-rings-are-quotients} below). More generally, we will see that every simple minimal solution ring for $M$ (i.e.~without the assumption on the constants) is a quotient of $U$. Conversely, in Cor.~\ref{cor:special-quotients-are-pv-rings} we show that every quotient of $U$ by a maximal $\C$-ideal $\m$ is a Picard-Vessiot ring if $(U/\m)^\C=k$.\\
Dropping the assumption $(U/\m)^\C=k$, however, one still has a simple solution ring $U/\m$ (by Prop.~\ref{prop:image-of-solution-rings}), but $U/\m$ may not be minimal. To see this, let $M=\1$. Then trivially $R:=\1$ is a Picard-Vessiot ring for $M$, and the only one, since it is contained in any other $\C$-algebra.\\
The universal solution ring for $M=\1$, however, is given by $U\isom \1\otimes_k k[x,x^{-1}]$. Hence, for every maximal ideal $I$ of $k[x,x^{-1}]$, $\m:=\iota(I)$ is a maximal $\C$-ideal of $U=\iota( k[x,x^{-1}])$ by Lemma~\ref{lemma:ideal-bijection::abstract}. But $U/\m\isom \iota(k[x,x^{-1}]/I)$ is only a minimal solution ring, if $k[x,x^{-1}]/I\isom k$, i.e. $U/\m\isom \1$.
\end{rem}

We continue with properties of quotients of $U$.

\begin{prop}\label{prop:properties-of-quotients-of-U}
Let $U$ be the universal solution ring for some dualizable $M\in \C$, and let $R$ be a quotient algebra of $U$. Then $\upsilon(R)$ is a finitely generated faithfully flat $\O_\X$-algebra. If in addition $R$ is a simple $\C$-algebra, then $R^\C$ is a finite field extension of $k$.
\end{prop}

\begin{proof}
Since $R$ is a quotient of $U$, it is a quotient of $T:=\Sym\Bigl( (M\otimes (\1^n)^\vee ) \oplus (\1^n\otimes M^\vee) \Bigr)$. Since $\upsilon(M)$ is finitely generated, $\upsilon(T)$ is a finitely generated $\O_\X$-algebra and therefore also $\upsilon(R)$ is a finitely generated $\O_\X$-algebra.

Since $M$ is dualizable, $\tenscat{M}$ is a Tannakian category (see Rem.~\ref{rem:dualizables-are-projective}), and $T$ is an ind-object of $\tenscat{M}$. Being a quotient of $T$, $R$ also is an ind-object of $\tenscat{M}$. Therefore by  \cite[Lemma 6.11]{pd:ct}, $\upsilon(R)$ is faithfully flat over $\O_\X$.

If in addition $R$ is simple, $\ell:=R^\C$ is a field. By exactness of $\iota$ and Prop.~\ref{prop:adjointness}(iii), we have a monomorphism
$\iota(\ell)\hookrightarrow R$, and hence by exactness of $\upsilon$, an inclusion of $\O_\X$-algebras
$\O_\X \otimes_k \ell=\upsilon(\iota(\ell))\hookrightarrow \upsilon(R)$.
After localising to some affine open subset of $\X$, we can apply Thm.~\ref{thm:abstract-algebra}, and obtain that $\ell$ is a finite extension of $k$.
\end{proof}

\begin{thm}\label{thm:simple-minimal-solution-rings-are-quotients}
Let $M$ be a dualizable object of $\C$, and let $U$ be the universal solution ring for $M$.
Then every simple minimal solution ring for $M$ is isomorphic to a quotient of the universal solution algebra $U$.
In particular, every Picard-Vessiot ring for $M$ is isomorphic to a quotient of $U$.
\end{thm}

\begin{proof}
Let $R$ be a simple minimal solution ring for $M$. Since $R$ is simple, $R^\C$ is a field, and therefore $(R\otimes M)^\C$ is a free $R^\C$-module.
Hence $R$ fulfills the assumptions of Theorem \ref{thm:univ-sol-ring}, and there is a morphism of $\C$-algebras $f:U\to R$.
As $(U\otimes M)^\C$ is a free $U^\C$-module, the image $f(U)$ is a solution ring by Prop.~\ref{prop:image-of-solution-rings}.
As $R$ is minimal, we obtain $f(U)=R$. Hence, $R$ is the quotient of $U$ by the kernel of $f$.
\end{proof}

\begin{cor}\label{cor:properties-of-simple-minimal-solution-rings}
Let $R\in \C$ be a simple minimal solution ring for some dualizable $M\in \C$.
Then $\upsilon(R)$ is a finitely generated faithfully flat $\O_\X$-algebra, and $R^\C$ is a finite field extension of $k$.
\end{cor}

\begin{proof}
This follows directly from Thm.~\ref{thm:simple-minimal-solution-rings-are-quotients} and Prop.~\ref{prop:properties-of-quotients-of-U}.
\end{proof}

\begin{prop}\label{prop:unique-pv-inside-simple-sol-ring}
Let $M$ be a dualizable object of $\C$, and let $R$ be a simple solution ring for $M$ with $R^\C=k$.
Then there is a unique Picard-Vessiot ring for $M$ inside $R$. This is the image of the universal solution ring $U$ under a morphism $f:U\to R$.
\end{prop}

\begin{proof}
As in the proof of Thm.~\ref{thm:simple-minimal-solution-rings-are-quotients}, $R$ fulfills the assumptions of Theorem \ref{thm:univ-sol-ring},
so there is a morphism of $\C$-algebras $f:U\to R$. By assumption on $R$, we have $\iota(R^\C)=\iota(k)=\1$,
and hence $\eps_R\otimes f=f:\1\otimes U=U\to R$. So by the second part of Theorem \ref{thm:univ-sol-ring}, the image $f(U)$ does not depend
on the choice of $f$. In particular, $f(U)$ (which is a solution ring by Prop.~\ref{prop:image-of-solution-rings}) is the unique minimal solution ring
inside $R$. It remains to show that $f(U)$ is a simple algebra.

Let $I\subseteq U$ be a maximal subobject in $\C_U$ (i.e.~an ideal of $U$), let $R':=U/I$ and let $g:U\to R'$ be the canoncial epimorphism. Furthermore, let $\m\in \C$ be a maximal ideal of $R'\otimes R$. Since $R$ and $R'$ are simple, the natural morphisms $R\to (R'\otimes R)/\m$ and $R'\to (R'\otimes R)/\m$ considered in $\C_R$ and $\C_{R'}$, respectively, are monomorphisms, and it suffices to show  that $\1\otimes f(U)\subseteq (R'\otimes R)/\m$  is simple.
\[
\xymatrix{
 U \ar^f[rr] \ar_g[d] && R \ar^{1\otimes \id_R}[d]\\
 R' \ar^(.4){\id_{R'}\otimes 1}[rr] && (R'\otimes R)/\m \\
}
\]

$g(U)=R'$ is simple by construction, and so is $g(U)\otimes \1\subseteq  (R'\otimes R)/\m$. By Theorem \ref{thm:univ-sol-ring}, we have $\iota(l)\cdot (g(U)\otimes \1)=\iota(l)\cdot (1\otimes f(U))$, where $l=\left((R'\otimes R)/\m\right)^\C$, and $l$ is a field, since $(R'\otimes R)/\m$ is simple. By Corollary \ref{cor:still-simple::abstract}, applied to the category $\C_{R'}$, $\iota(l)\cdot (g(U)\otimes \1)$ is also simple, i.e.~$\iota(l)\cdot (\1\otimes f(U))$ is simple.
Since, $\iota(l)\cdot (\1\otimes f(U))\cong l\otimes_{k} f(U)$ is a faithfully flat extension of $f(U)$, this implies that $f(U)$ is also simple.
\end{proof}

\begin{rem}
The previous proposition ensures the existence of Picard-Vessiot rings in special cases. For example, in the differential setting over
e.g.~$F=\mathbb{C}(t)$, if $x$ is a point which is non-singular for the differential equation, then one knows that the ring of  holomorphic functions
on a small disc around that point is a solution ring for the equation.
Hence, there exists a Picard-Vessiot ring (even unique) for the corresponding differential module inside this ring of holomorphic functions.\\
Similarly, in the case of rigid analytically trivial pre-$t$-motives (which form a special case of the difference setting) the field of fractions of a 
given ring of restricted power series  is a simple solution ring for all these modules (cf.~\cite{mp:tdadmaicl}).
\end{rem}

\begin{cor}\label{cor:special-quotients-are-pv-rings}
Let $M\in \C$ be dualizable, and let $\m$ be a maximal $\C$-ideal of the universal solution ring $U$ for $M$ such that $(U/\m)^\C=k$. Then $U/\m$ is a Picard-Vessiot ring for $M$.
\end{cor}

\begin{proof}
By Prop.~\ref{prop:image-of-solution-rings}, $U/\m$ fulfills the conditions of $R$ in the previous propostion. Hence, the image of the morphism $U\to U/\m$ (which clearly is $U/\m$) is a Picard-Vessiot ring. 
\end{proof}

\begin{cor}\label{cor:pv-rings-isom-over-finite-ext}
Let $M\in \C$ be dualizable, and let $R$ and $R'$ be two simple minimal solution rings for $M$. Then there exists a finite field extension $\ell$ of $k$ containing $R^\C$ and $(R')^\C$ such that $R\otimes_{R^\C} \ell\isom R'\otimes_{(R')^\C} \ell$.
\end{cor}

\begin{proof}
As in the proof of the previous theorem, let $f:U\to R$ and $g:U\to R'$ be epimorphisms of $\C$-algebras whose existence is guaranteed by Thm.~\ref{thm:simple-minimal-solution-rings-are-quotients}. Let $\m$ be a maximal $\C$-ideal of $R'\otimes R$, and let $\ell:=\left(R'\otimes R/\m\right)^\C$. 
Then $R'$ and $R$ embed into $R'\otimes R/\m$ and hence $(R')^\C$ and $R^\C$ both embed into $\ell$. Furthermore
 by Thm.~\ref{thm:univ-sol-ring}, the subrings $\iota(\ell)(g(U)\otimes 1)$ and 
$\iota(\ell)(1\otimes f(U))$ are equal. As $\ell$ contains both $R^\C$ and $(R')^\C$, one has $\iota(\ell)(g(U)\otimes 1)=\iota(\ell)(R'\otimes 1)\isom R'\otimes_{(R')^\C} \ell$ and
$\iota(\ell)(1\otimes f(U))\isom R\otimes_{R^\C} \ell$. Hence, $R'\otimes_{(R')^\C} \ell \isom R\otimes_{R^\C} \ell$. As in the proof of Prop.~\ref{prop:properties-of-quotients-of-U}, one shows that $\ell$ is indeed finite over $k$.
\end{proof}

\begin{thm}\label{thm:existence-of-pv-ring}
Let $M\in \C$ be dualizable. Then there exists a Picard-Vessiot ring for $M$ up to a finite field extension of $k$, i.e.~there exists a finite field extension $\ell$ of $k$ and a $\C_{\iota(\ell)}$-algebra $R$ such that $R$ is a PV-ring for $M_{\iota(\ell)}\in \C_{\iota(\ell)}$.
\end{thm}

\begin{proof}
Let $U$ be the universal solution ring for $M$, and let $\m\subset U$ be a maximal $\C$-ideal of $U$. Then $R:=U/\m$ is a simple solution ring for $M$ by Prop.~\ref{prop:image-of-solution-rings}, and $\ell:=R^\C$ is a finite field extension of $k$ by Prop.~\ref{prop:properties-of-quotients-of-U}.\\
Considering now $M_{\iota(\ell)}\in \C_{\iota(\ell)}$, and $R$ as an algebra in $\C_{\iota(\ell)}$ via  $\eps_R:\iota(R^\C)=\iota(\ell)\to R$, we obtain that $R$ is a simple solution ring for $M_{\iota(\ell)}$ with $R^\C=\ell$.
Hence by Prop.~\ref{prop:unique-pv-inside-simple-sol-ring}, with $k$ replaced by $\ell$ (and $\C$ by $\C_{\iota(\ell)}$ etc.), there is a unique Picard-Vessiot ring for $M_{\iota(\ell)}$ inside $R$. Indeed also by Prop.~\ref{prop:unique-pv-inside-simple-sol-ring}, this Picard-Vessiot ring is $R$ itself, since the canonical morphism $\iota(\ell)\otimes U\to R$ is an epimorphism, and $\iota(\ell)\otimes U$ is easily seen to be the universal solution ring for $M_{\iota(\ell)}$.
\end{proof}

\section{Picard-Vessiot rings and fibre functors}\label{sec:pv-rings-and-fibre-functors}
Throughout this section, we fix a dualizable object $M\in \C$.
Recall that we denote by $\tenscat{M}$ the strictly full tensor subcategory of $\C$ generated by $M$ and $M^\vee$, i.e.~the full subcategory of $\C$ containing all objects isomorphic to subquotients of direct sums of objects $M^{\otimes n}\otimes (M^\vee)^{\otimes m}$ for $n,m\geq 0$. 

In this section we consider the correspondence between Picard-Vessiot rings $R$ for $M$ and fibre functors $\omega:\tenscat{M}\to \vect_k$.
The main result is Thm.~\ref{thm:pv-rings-equiv-to-fibre-functors} which states that there is a bijection between their isomorphism classes. This generalises \cite[Thm.~3.4.2.3]{ya:dnctgdd} to our abstract setting.

\begin{prop}\label{prop:fibre-functor-associated-to-pv-ring}
Assume $R$ is a Picard-Vessiot ring for $M$. Then the functor
$$\omega_R:\tenscat{M}\to \vect_k, N\mapsto (R\otimes N)^\C$$
is an exact faithful tensor-functor, i.e.~a fibre functor.\\
We call the fibre functor $\omega_R$ the {\markdef fibre functor associated to $R$}.
\end{prop}

\begin{proof}
By definition of a Picard-Vessiot ring, the morphism $\eps_{M_R}:R\otimes_k (R\otimes M)^\C \to R\otimes M$ is an isomorphism. Hence, by Prop.~\ref{prop:subcat-of-trivial-modules}, $\eps_{N_R}$ is an isomorphism for all $N\in \tenscat{M}$.

Recall $R\otimes_k (R\otimes N)^\C=\iota_R((N_R)^\C)=\iota_R(\omega_R(N))$ for all $N$.

As $\upsilon(R)$ is faithfully flat over $\O_\X=\upsilon(\1)$ by Cor.~\ref{cor:properties-of-simple-minimal-solution-rings}, the functor $N\mapsto R\otimes N$ is exact and faithful. Hence, given a short exact sequence $0\to N'\to N\to N''\to 0$ in $\tenscat{M}$, the sequence
$$0\to R\otimes N'\to R\otimes N\to R\otimes N''\to 0$$
is exact, and $R\otimes N=0$ if and only if $N=0$.
Using the isomorphisms  $\eps_{N_R}$ etc. the sequence
$$0\to R\otimes_k  \omega_R(N')\to R\otimes_k  \omega_R(N)\to R\otimes_k  \omega_R(N'')\to 0$$
is exact.
As $\iota_R$ is exact and faithful, this implies that
$$0\to \omega_R(N')\to  \omega_R(N)\to  \omega_R(N'')\to 0$$
is exact. Furthermore, $ \omega_R(N)=0$ if and only if $ R\otimes_k  \omega_R(N)=0$
 if and only if $ R\otimes N=0$ if and only if $N=0$.
 
 It remains to show that $\omega_R$ is a tensor-functor which is already done by showing that 
 $\eps_{(N\otimes N')_R}$ is an isomorphism if $\eps_{N_R}$ and $\eps_{N'_R}$ are.
\end{proof}

Given a fibre functor $\omega:\tenscat{M} \to \vect_k$, we want to obtain a Picard-Vessiot ring associated to $\omega$.\\
Apparently, this Picard-Vessiot ring is already given in the proof of \cite[Thm.~3.2]{pd-jsm:tc}, although the authors don't claim that it is a Picard-Vessiot ring. We will recall the construction to be able to prove the necessary facts:\\
For $N\in \tenscat{M}$, one defines $P_N$ to be the largest subobject of $N\otimes_k \omega(N)^\vee$ such that for all $n\geq 1$ and all subobjects $N'\subseteq N^n$, the morphism
$$P_N\to N\otimes_k \omega(N)^\vee\xrightarrow{\text{diag}}N^n\otimes_k \omega(N^n)^\vee
\to N^n\otimes_k \omega(N')^\vee$$
factors through $N'\otimes_k \omega(N')^\vee$.\\
For monomorphisms $g:N'\to N$ and epimorphisms $g:N\to N'$, one obtains morphisms $\phi_{g}:P_N\to P_{N'}$, and therefore
$$R_\omega:= \varinjlim\limits_{N\in\tenscat{M}} P_N^\vee\in{\rm Ind}(\tenscat{M})\subseteq \C$$
is welldefined. The multiplication $\mu_{R_\omega}:R_\omega\otimes R_\omega\to R_\omega$ is induced by the natural morphisms $P_{N\otimes L}\to P_N\otimes P_L$ via dualizing and taking inductive limits.

\begin{lemma}\label{lem:R-omega-representing}
The functor $\C{\rm -}\alg \to \sets$ which associates to each $\C$-algebra $R'$ the set of natural tensor-transformations from the functor $R'\otimes (\iota\circ \omega):\tenscat{M}\to \C_{R'}$ to the functor $R'\otimes \id_{\tenscat{M}}:\tenscat{M}\to \C_{R'}$ is represented by the $\C$-algebra $R_\omega$, i.e.~there is a natural bijection between the natural transformations $R'\otimes (\iota\circ \omega)\to R'\otimes \id_{\tenscat{M}}$ of tensor functors and the morphisms of $\C$-algebras
$R_\omega\to R'$.
\end{lemma}

\begin{proof} 
Let $R'$ be a $\C$-algebra, and let $\alpha$ be a natural transformation not necessarily respecting the tensor structure. Then for every
 $N\in \tenscat{M}$ one has a morphism 
\begin{eqnarray*}
\alpha_N & \in & \Mor_{\C_{R'}}(R'\otimes \iota(\omega(N)), R'\otimes N)
\simeq \Mor_{\C}(\iota(\omega(N)), R'\otimes N) \\
&\simeq & \Mor_{\C}(\1, R'\otimes N\otimes \iota(\omega(N))^\vee)= (R'\otimes N\otimes \iota(\omega(N))^\vee)^\C
\end{eqnarray*}


It is straight forward to check that such a collection of morphisms
$(\alpha_N)_{N}$ where $\alpha_N\in  \Mor_{\C}(\1, R'\otimes N\otimes \iota(\omega(N))^\vee)$ defines a natural transformation if and only if 
$\alpha_N\in  \Mor_{\C}(\1, R'\otimes P_N)$ for all $N$, and $\alpha_{N'}=(\id_{R'}\otimes \phi_{g})\circ \alpha_N$ whenever $\phi_{g}:P_N\to P_{N'}$ is defined.\\
On the other hand, one has
\begin{eqnarray*}
\Mor_{\C}(R_\omega,R') &=& \Mor_{\C}(\varinjlim\limits_{N\in\tenscat{M}} P_N^\vee, R') \\
&=& \varprojlim\limits_{N\in\tenscat{M}}\Mor_{\C}(P_N^\vee, R') \simeq
\varprojlim\limits_{N\in\tenscat{M}}\Mor_{\C}(\1,R'\otimes P_N)
\end{eqnarray*}
Hence, giving such a compatible collection of morphisms $\alpha_N$ is equivalent to giving a $\C$-morphism $R_\omega\to R'$.\\
It is also not hard to check that the natural transformations that respect the tensor structure correspond to morphisms of $\C$-algebras $R\to R'$ under this identification. 
\end{proof}

Before we show that $R_\omega$ is a simple solution ring for $M$, we need some more results from \cite{pd-jsm:tc} resp. from \cite{pd:ct}:\\
As $\omega$ has values in $k$-vector spaces, $\tenscat{M}$ together with $\omega$ is a neutral Tannakian category (see \cite{pd:ct}), and therefore equivalent to the category of representations of the algebraic group scheme $G=\underline{\Aut}^\otimes(\omega)$.\\
This also induces an equivalence of their ind-categories, and $R_\omega$ corresponds to the group ring $k[G]$ with the right regular representation (cf.~proof of \cite[Theorem 3.2]{pd-jsm:tc}).

%

\begin{prop}\label{prop:pv-ring-associated-to-fibre-functor}
The object $R_\omega\in {\rm Ind}(\tenscat{M})\subseteq \C$ associated to $\omega$ is a
simple solution ring for $M$, and satisfies $(R_\omega)^\C=k$.
\end{prop}

\begin{rem}\label{rem:pv-ring-associated-to-fibre-functor}
By Prop.~\ref{prop:unique-pv-inside-simple-sol-ring}, $R_\omega$ therefore contains a unique Picard-Vessiot ring for $M$. This Picard-Vessiot ring will be called the {\markdef PV-ring associated to $\omega$}.
Indeed, $R_\omega$ is already minimal and hence a Picard-Vessiot ring itself. This will be seen at the end of the proof of Thm.~\ref{thm:pv-rings-equiv-to-fibre-functors}.
There is also a way of directly showing that $R_\omega$ is isomorphic to a quotient of the universal solution ring for $M$ which would also imply that $R_\omega$ is a PV-ring (cf.~Cor.~\ref{cor:special-quotients-are-pv-rings}). But we don't need this here, so we will omit it.
\end{rem}

\begin{proof}
As $\omega$ defines an equivalence of categories $\tenscat{M}\to {\rm Rep}_k(G)$ (and also of their ind-categories), and  $\omega(R_\omega)=k[G]$, one obtains
$$(R_\omega)^{\C}=\Mor_\C(\1,R_\omega)\simeq \Hom_G(k,k[G])=k[G]^G=k.$$
For showing that $R_\omega$ is simple, let $I\ne R_\omega$ be an ideal of $R_\omega$ in $\C$. We even have $I\in {\rm Ind}(\tenscat{M})$, as it is a subobject of $R$. By the equivalence of categories $\omega(I)$ belongs to ${\rm Ind}({\rm Rep}_k(G))$, and $\omega(I)$ is an ideal of $\omega(R_\omega)=k[G]$. But $k[G]$ does not have non-trivial $G$-stable ideals. Hence, $\omega(I)=0$, and therefore $I=0$.

As seen in Lemma \ref{lem:R-omega-representing}, $\id_{R_\omega}\in \Mor_{\C}(R_\omega,R_\omega)$ induces a natural transformation $\alpha:R_\omega\otimes (\iota\circ \omega) \to R_\omega\otimes \id_{\tenscat{M}}$, in particular it induces a $\C_{R_\omega}$-morphism $\alpha_M:R_\omega\otimes \iota(\omega(M))\to R_\omega\otimes M$.
By \cite[Prop.~1.13]{pd-jsm:tc}, such a natural transformation is an equivalence, as $\tenscat{M}$ is rigid\footnote{Rigidity of the target category which is assumed in loc.~cit.~is not needed. See also \cite[Prop.~1.1]{ab:ttnc}.}. Therefore, the morphism $\alpha_M$ is an isomorphism. As $R_\omega\otimes \iota(\omega(M))=\iota_{R_\omega}(\omega(M))$, Lemma \ref{lemma:when-eps-is-iso} implies that $\eps_{M_R}$ is an isomorphism.\\
Hence, $R_\omega$ is a solution ring for $M$.


%

\end{proof}

\begin{thm}\label{thm:pv-rings-equiv-to-fibre-functors}
Let $M\in \C$ be dualizable,
and let $\ell$ be a field extension of $k$. Then there is a bijection between isomorphism classes of Picard-Vessiot rings $R$ for $M_{\iota(\ell)}$ over $\wtilde{\1}:=\iota(\ell)$ and isomorphism classes of fibre functors $\omega$ from $\tenscat{M_{\iota(\ell)}}$ into $\ell$-vector spaces.\\
This bijection is induced by $R\mapsto \omega_R$ and $\omega\mapsto (\text{PV-ring inside }R_\omega)$ given in Prop.~\ref{prop:fibre-functor-associated-to-pv-ring} and Rem.~\ref{rem:pv-ring-associated-to-fibre-functor}, respectively.
\end{thm}

\begin{proof}
Clearly isomorphic Picard-Vessiot rings give rise to isomorphic fibre functors and isomorphic fibre functors give rise to isomorphic Picard-Vessiot rings. Hence, we only have to show that the maps are inverse to each other up to isomorphisms.\\
By working directly in the category $\C_{\iota(\ell)}$ we can assume that $\ell=k$.

On one hand, for given $\omega$ and corresponding PV-ring $R$, one has natural isomorphisms 
$$\iota_R(\omega(N))=R\otimes_k \omega(N)\to N_R$$
(see proof of Prop.~\ref{prop:pv-ring-associated-to-fibre-functor}). By adjunction these correspond to natural isomorphisms 
$$\lambda_N:\omega(N)\cong(N_R)^\C=\omega_R(N),$$ 
i.e.~the functors $\omega$ and $\omega_R$ are isomorphic.

Conversely, given a Picard-Vessiot ring $R$ and associated fibre functor $\omega_R$, let $R_\omega$ be the simple solution ring constructed above.\\
As $\iota_R=R\otimes \iota$ and $(N_R)^{\C_R}=\omega_R(N)$ for all $N\in \tenscat{M}$, the natural isomorphisms $\eps_{N_R}:\iota_R\left( (N_R)^{\C_R} \right) \to N_R$ form a natural transformation
$R\otimes (\iota\circ \omega_R) \to R\otimes \id_{\tenscat{M}}$.
By Lemma \ref{lem:R-omega-representing}, this natural transformation corresponds to a morphism of $\C$-algebras $\varphi:R_\omega\to R$. As $R_\omega$ is a simple $\C$-algebra, $\varphi$ is a monomorphism. But $R$ is a minimal solution ring, and hence $\varphi$ is even an isomorphism. Therefore, $R_\omega$ is isomorphic to $R$ and already minimal, i.e.~$R_\omega$ is a Picard-Vessiot ring itself.
\end{proof}

\section{Galois group schemes}\label{sec:galois-groups}

Given a dualizable object $M\in \C$ and a Picard-Vessiot ring $R$ for $M$, one considers the group functor
$$\underline{\Aut}_{\C-\text{alg}}(R): \alg_k \to \groups$$
which associates to each $k$-algebra $D$ the group of automorphisms of $R\otimes_k D$ as an algebra in $\C_{\iota(D)}$, i.e.~the subset of $\Mor_{\C_{\iota(D)}}(R\otimes_k D,R\otimes_k D)$ consisting of all isomorphisms which are compatible with the algebra structure of $R\otimes_k D$.\\
This functor is called the {\markdef Galois group of $R$} over $\1$.

On the other hand, given a fibre functor $\omega:\tenscat{M}\to \vect_k$, one considers 
the group functor
$$\underline{\Aut}^\otimes(\omega): \alg_k \to \groups$$
which associates to each $k$-algebra $D$ the group of natural automorphisms of the functor
$D\otimes_k \omega: N\mapsto D\otimes_k\omega(N)$.\\
As $\tenscat{M}$ together with the fibre functor $\omega$ is a neutral Tannakian category, this group functor is called the {\markdef Tannakian Galois group} of $(\tenscat{M},\omega)$. In \cite{pd:ct} it is shown that this group functor is indeed an algebraic group scheme.

The aim of this section is to show that both group functors are isomorphic algebraic group schemes if $\omega=\omega_R$ is the fibre functor associated to $R$.

\medskip

We start by recalling facts about group functors, (commutative) Hopf-algebras and affine group schemes. All of this can be found in \cite{ww:iags}.

A group functor $ \alg_k \to \groups$ is an affine group scheme over $k$ if it is representable by a commutative algebra over $k$. This commutative algebra then has a structure of a Hopf-algebra. The group functor is even an algebraic group scheme (i.e.~of finite type over $k$) if the corresponding Hopf-algebra is finitely generated.\\
The category of commutative Hopf-algebras over $k$ and the category of affine group schemes over $k$ are equivalent. This equivalence is given by taking the spectrum of a Hopf-algebra in one direction and by taking the ring of regular functions in the other direction.\\
For a Hopf-algebra $H$ over $k$, and corresponding  affine group scheme 
$\G:=\spec(H)$, the category $\Comod(H)$ of right comodules of $H$ and the category $\Rep(\G)$ of representations of $\G$ are equivalent.
This equivalence is given by attaching to a comodule $V$ with comodule map $\rho:V\to V\otimes_k H$ the following representation $\varrho:\G\to \End(V)$ of $\G$:
For any $k$-algebra $D$ and $g\in \G(D)=\Hom_{k\rm{-alg}}(H,D)$, the endomorphism $\varrho(g)$ on $V\otimes_k D$ is the $D$-linear extension of
$$g\circ \rho:V\to  V\otimes_k H\to V\otimes_k D.$$
On the other hand, for any representation $\varrho:\G\to \End(V)$, the universal element $\id_H\in \Hom_{k\rm{-alg}}(H,H)=\G(H)$ gives a $H$-linear homomorphism $\varrho(\id_H):V\otimes_k H\to V\otimes_k H$, and its restriction to $V\otimes 1$ is the desired comodule map $\rho:V\to V\otimes_k H$.

\medskip

For showing that the group functors $\underline{\Aut}_{\C-\text{alg}}(R)$ and
$\underline{\Aut}^\otimes(\omega_R)$ are isomorphic algebraic group schemes, we show that they are both represented by the $k$-vector space $H:=(R\otimes R)^\C=\omega_R(R)$. The next lemma shows that $H$ is a finitely generated (commutative) $k$-Hopf-algebra, and hence $\spec(H)$ is an algebraic group scheme over $k$.

\begin{rem}
This fact is shown for differential modules over algebraically closed constants in \cite[Thm.~2.33]{mvdp-mfs:gtlde}, and for t-motives in \cite[Sections 3.5-4.5]{mp:tdadmaicl}.
\end{rem}

\begin{lemma}\label{lemma:H-is-Hopf-algebra}
Let $R$ be a PV-ring for $M$ and $H:=\omega_R(R)=(R\otimes R)^\C$.
\begin{enumerate}
\item The morphism $\eps_{R_R}:R\otimes_k H \to R_R=R\otimes R$ is an isomorphism in $\C_R$ (with $R$-module structure on $R\otimes R$ given on the first factor).
\item $H$ is a finitely generated commutative $k$-algebra where the structure maps $u_H:k\to H$ (unit), $\mu_H:H\otimes_k H\to H$ (multiplication) are given by
$$u_H:=\omega_R(u_R) \quad \text{and} \quad \mu_H:=\omega_R(\mu_R),$$
respectively.
\item The $k$-algebra $H$ is even a Hopf-algebra where the structure maps $c_H:H\to k$ (counit), $\Delta:H\to H\otimes_k H$ (comultiplication) and $s:H\to H$ (antipode) are given as follows:
Counit and antipode are given by
$$c_H:=(\mu_R)^\C  \quad \text{and} \quad  s:=(\tau)^\C,$$
respectively, where $\tau\in \Mor_\C(R\otimes R, R\otimes R)$ denotes the twist morphism.
The comultiplication is given by
$$\Delta:=\omega_R\bigl( \eps_{R_R}^{-1}\circ (u_R\otimes \id_R)\bigr)\ 
\footnote{Hence, $\Delta$ is the image under $\omega_R$ of the morphism
$R\xrightarrow{u_R \otimes \id_R} R\otimes R  \xrightarrow{\eps_{R_R}^{-1}} R\otimes_k H$}$$
\end{enumerate}
\end{lemma}

\begin{rem}
The definition of $\Delta$ might look strange. Compared to other definitions (e.g.~in \cite[Sect.~2]{mt:haapvt}), where $\Delta$ is the map on constants/invariants induced by the map 
$R\otimes R\to R\otimes R\otimes R, a\otimes b\mapsto a\otimes 1\otimes b$, one might think that
$\Delta$ should be defined as $(\id_R\otimes u_R\otimes \id_R)^\C=\omega_R(u_R\otimes \id_R)$.
The reason for the difference is that in \cite{mt:haapvt} and others, one uses
$(R\otimes R)\otimes_R  (R\otimes R)\isom R\otimes R\otimes R$ with right-$R$-module structure on the left tensor factor $(R\otimes R)$ and left-$R$-module structure on the right tensor factor $(R\otimes R)$.\\
In our setting, however, we are always using left-$R$-modules. In particular, the natural isomorphism $\omega_R(R)\otimes_k \omega_R(R)\to \omega_R(R\otimes R)$ reads as
$$\Mor_{\C_R}(R,R\otimes R)\otimes_k \Mor_{\C_R}(R,R\otimes R)\to  \Mor_{\C_R}(R,R\otimes R\otimes R)$$
where the left hand side is isomorphic to $\Mor_{\C_R}(R,(R\otimes R)\otimes_R (R\otimes R))$. But here, this is the tensor product of left-$R$-modules.\\
The additional $\eps_{R_R}^{-1}$ in the definition of $\Delta$ solves the problem. It is also implicitly present in the identification
$H\otimes_k H\isom (R\otimes R\otimes R)^\C$ in \cite{mt:haapvt} (cf.~proof of Lemma 2.4(b) loc.~cit.).
\end{rem}

\begin{proof}[Proof of Lemma \ref{lemma:H-is-Hopf-algebra}]
As $R$ is an object of ${\rm Ind}(\tenscat{M})$, part (i) follows from Prop.~\ref{prop:subcat-of-trivial-modules}.
As $\omega_R$ is a tensor functor, it is clear that the structure of a commutative algebra of $R$ induces a structure of a commutative algebra on $\omega_R(R)=H$ via the maps $u_H$ and $\mu_H$ defined in the lemma. 
As in the proof of Prop.~\ref{prop:properties-of-quotients-of-U}, one verifies that $H=\omega_R(R)$ is finitely generated as $k$-algebra.\\
Part (iii) is obtained by checking that the necessary diagrams commute. 
We only show that $\Delta$ is coassociative, i.e.~that 
$(\Delta\otimes_k \id_H)\circ \Delta=(\id_H\otimes_k \Delta)\circ \Delta$, and leave the rest to the reader.

As $\Delta=\omega_R\bigl( \eps_{R_R}^{-1}\circ (u_R\otimes \id_R)\bigr)$,
$\Delta\otimes_k \id_H=\omega_R\bigl( (\eps_{R_R}^{-1}\otimes_k \id_H) \circ (u_R\otimes \id_R\otimes_k \id_H)\bigr)$ and
$\id_H\otimes_k \Delta=\omega_R( \id_R\otimes_k \Delta)$, it suffices to show that the morphisms
$$ (\eps_{R_R}^{-1}\otimes_k \id_H) \circ (u_R\otimes \id_R\otimes_k \id_H) \circ \eps_{R_R}^{-1}\circ (u_R\otimes \id_R) \quad \text{and}$$
$$( \id_R\otimes_k \Delta) \circ  \eps_{R_R}^{-1}\circ (u_R\otimes \id_R)$$
are equal. This is seen by showing that the following diagram commutes:
$$\xymatrix@C+25pt{
R \ar[r]^{u_R\otimes \id_R} \ar[d]_{u_R\otimes \id_R} &
R\otimes R \ar[r]^{\eps_{R_R}^{-1}} \ar[d]_(0.45){u_R\otimes \id_{R\otimes R}}   & 
R\otimes_k H \ar[d]^{u_R\otimes \id_{R\otimes_k H}} \\
R\otimes R \ar[r]^(0.45){\id_R\otimes u_R\otimes \id_R}\ar[d]_{\eps_{R_R}^{-1}} &
R\otimes R\otimes R \ar[r]^{\id_R\otimes \eps_{R_R}^{-1}} &
R\otimes R\otimes_k H \ar[d]^{\eps_{R_R}^{-1}\otimes_k \id_H} \\
R\otimes_k H \ar[rr]^{\id_R\otimes_k \Delta=\iota_R(\Delta) }&
&
R\otimes_k H\otimes_k H
}$$
Obviously the upper squares commute. Let $\delta:=\eps_{R_R}^{-1}\circ (u_R\otimes \id_R)$. Then the middle horizontal morphism equals $\id_R\otimes \delta$ and the lower horizontal morphism is $\iota_R(\Delta)=\iota_R((\id_R\otimes \delta)^{\C_R})$. As $\eps$ is a natural transformation 
$\iota_R\circ ()^{\C_R}\to \id_{\C_R}$, and as
$\eps_{R_R}^{-1}\otimes_k \id_H=\eps_{(R\otimes_k H)_R}^{-1}$, also the lower square commutes.

\end{proof}

\begin{thm}\label{thm:Aut-R-represented-by-H}
Let $R$ be a PV-ring for $M$. Then the group functor
$$\underline{\Aut}_{\C-\text{alg}}(R): \alg_k \to \groups$$
is represented by the Hopf-algebra $H=\omega_R(R)=(R\otimes R)^\C$.
Furthermore $\spec(\upsilon(R))$ is a torsor of $\underline{\Aut}_{\C-\text{alg}}(R)$ over $X$.
\end{thm}

\begin{proof}
This is shown similar to \cite[Prop.10.9]{am:gticngg} or \cite{td:tipdgtfrz}.
One has to use that $$\delta: R\xrightarrow{u_R \otimes \id_R} R\otimes R  \xrightarrow{\eps_{R_R}^{-1}} R\otimes_k H$$
defines a right coaction of $H$ on $R$. The property of a right coaction, however, is given by the commutativity of the diagram in the proof of Lemma~\ref{lemma:H-is-Hopf-algebra}.

The torsor property is obtained by the isomorphism $\upsilon(\eps_{R_R}^{-1}):\upsilon(R)\otimes_{\O_\X} \upsilon(R)\to \upsilon(R)\otimes_k H$.
\end{proof}

\begin{thm}\label{thm:H-acting-on-omega_R}
Let $R$ be a PV-ring for $M$ and $H=\omega_R(R)$.
\begin{enumerate}
\item For all $N\in \tenscat{M}$,  $\rho_N:\omega_R(N)\to H\otimes_k \omega_R(N)$ given by
$$\rho_N:=\omega_R\bigl( \eps_{N_R}^{-1}\circ (u_R\otimes \id_{N}) \bigr)\ 
\footnote{The map $\eps_{N_R}^{-1}\circ (u_R\otimes \id_{N})$ is a morphism in $\C$: $N\to R\otimes N \to R\otimes_k \omega_R(N)$}$$
defines a left coaction of $H$ on $\omega_R(N)$.
\item The collection $\rho:=(\rho_N)_{N\in \tenscat{M}}$ is a natural transformation of tensor functors
$\omega_R\mapsto H\otimes_k \omega_R$, where $H\otimes_k \omega_R$ is a functor $\tenscat{M}\to \Mod_H$.
\end{enumerate}
\end{thm}

\begin{rem}
By going to the inductive limit one also gets a map $\rho_R:\omega_R(R)\to H\otimes_k \omega_R(R)$.
This map is nothing else then the comultiplication $\Delta:H\to H\otimes_k H$.
\end{rem}

\begin{proof}[Proof of Thm.~\ref{thm:H-acting-on-omega_R}]
Part (i) is proven in the same manner as the coassociativity of $\Delta$. For proving the second part, recall that $\eps$ is a natural transformation. Hence, for every morphism $f:N\to N'$ the diagram
$$\xymatrix@C+10pt{
N \ar[r]^(0.4){u_R \otimes \id_N} \ar[d]_{f} &
R\otimes N \ar[r]^(0.4){\eps_{N_R}^{-1}}  \ar[d]_{\id_R\otimes f}& 
R\otimes_k \omega_R(N) \ar[d]^{\iota_R((\id_R\otimes f)^\C)} \\
N' \ar[r]^(0.45){u_R \otimes \id_{N'}} & R\otimes N' \ar[r]^(0.4){\eps_{N'_R}^{-1}} & R\otimes_k \omega_R(N')
}$$
commutes. As $\iota_R((\id_R\otimes f)^\C)=\id_R\otimes_k \omega_R(f)$, applying $\omega_R$ to the diagram gives the desired commutative diagram for $\rho$ being a natural transformation. Compatibility with the tensor product is seen in a similar way.
\end{proof}

\begin{thm}\label{thm:Aut-omega_R-represented-by-H}
Let $R$ be a PV-ring for $M$ and $H=\omega_R(R)$. Then the group functor $$\underline{\Aut}^\otimes(\omega_R):\alg_k \to \groups$$
is represented by the Hopf-algebra $H$.\footnote{As shown in the following proof, the representing Hopf-algebra naturally is the coopposite Hopf-algebra $H^{\rm cop}$ of $H$. However, the antipode $s$ is an isomorphism of Hopf-algebras $s:H\to H^{\rm cop}$, hence $H^{\rm cop}\isom H$.}
\end{thm}

\begin{proof}
As $\rho:=(\rho_N)_{N\in \tenscat{M}}$  defines a left coaction of $H$ on the functor $\omega_R$ by natural transformations, one obtains a right action of $\spec(H)$ on $\omega_R$. Composing with the antipode (i.e. taking inverse group elements), one therefore gets a homomorphism of group functors
$$\varphi:\spec(H)\to \underline{\Aut}^\otimes(\omega_R).$$
Explicitly, for any $k$-algebra $D$ and $h\in H(D)=\Hom_{k-\rm{alg}}(H,D)$, one defines $\varphi(h)\in \underline{\Aut}^\otimes(\omega_R)(D)=\Aut^\otimes(D\otimes_k \omega_R)$ as the natural transformation which for $N\in\tenscat{M}$ is the $D$-linear extension of the composition 
$$\omega_R(N)\xrightarrow{\rho_N} H\otimes_k \omega_R(N)
\xrightarrow{s\otimes \id_{\omega_R(N)}} H\otimes_k \omega_R(N)
\xrightarrow{h\otimes \id_{\omega_R(N)}}D\otimes_k \omega_R(N).$$

For showing that the homomorphism $\varphi$ is indeed an isomorphism, we give the inverse map:\\
For any $k$-algebra $D$ and $g\in \underline{\Aut}^\otimes(\omega_R)(D)$, one has the homomorphism $g_R\in \End_D(D\otimes_k \omega_R(R))=\End_D(D\otimes_k H)$, and one defines $\psi(g)\in H(D)$ as the composition
$$H\xrightarrow{s} H\xrightarrow{u_D\otimes \id_H} D\otimes_k H \xrightarrow{g_R}D\otimes_k H  \xrightarrow{\id_D\otimes c_H}D.$$

It is a straight forward calculation to check that $\psi(g)$ is indeed a homomorphism of $k$-algebras and that $\varphi$ and $\psi$ are inverse to each other. 
\end{proof}

\begin{cor}\label{cor:auts-are-isomorphic}
The affine group schemes $\underline{\Aut}_{\C\text{-}\alg}(R)$ and $\underline{\Aut}^\otimes(\omega_R)$ are isomorphic.
\end{cor}

\begin{proof}
By Thm.~\ref{thm:Aut-R-represented-by-H} and Thm.~\ref{thm:Aut-omega_R-represented-by-H} both functors are represented by the Hopf-algebra $H=\omega_R(R)$.
\end{proof}

%

\section{Galois correspondence}\label{sec:galois-correspondence}

In this section we will establish a Galois correspondence between subalgebras of a PV-ring and closed subgroups of the corresponding Galois group. As in \cite{am:pvtdsr}, the Galois correspondence will only take into account subalgebras which are PV-rings themselves on the one hand, and
normal subgroups on the other.

\medskip

We start by recalling facts about sub-Hopf-algebras and closed subgroup schemes which can be found in \cite{ww:iags}.

In the equivalence of affine group schemes and Hopf-algebras, closed subgroup schemes correspond to Hopf-ideals, and closed normal subgroup schemes correspond to so called normal Hopf-ideals.
As there is a correspondence between closed normal subgroup schemes and factor group schemes of $\G$ by taking the cokernel and the kernel, respectively, there is also a correspondence between normal Hopf-ideals and sub-Hopf-algebras (\cite[Thm.~4.3]{mt:cbhisha}). This correspondence is given by
$$I\mapsto H(I):=\Ker\left(H\xrightarrow{\Delta - \id_H\otimes u_H}H\otimes_k H\to H\otimes_k (H/I)\right),$$
for a normal Hopf-ideal $I$, and  by
$$H'\mapsto (H')^+H,$$
for a sub-Hopf-algebra $H'$, where $(H')^+$ is defined to be the kernel of the counit $c_{H'}:H'\to k$.

Furthermore, for a sub-Hopf-algebra $H'\subseteq H$, the category $\Comod(H')$ embeds into $\Comod(H)$ as a full subcategory.


\begin{thm}\label{thm:galois-correspondence}
Let $M\in \C$ be dualizable, $R$ a PV-ring for $M$ (assuming it exists), $\omega=\omega_R$ the corresponding fibre functor, $H=\omega_R(R)$, and
$\G=\spec(H)=\underline{\Aut}_{\C\text{-}\alg}(R)=\underline{\Aut}^\otimes(\omega)$ the corresponding Galois group. 
Then there is a bijection between
$$\fT:=\{ T \in \C\text{-}\alg \mid T\subseteq R \text{ is PV-ring for some }N\in \tenscat{M} \}$$
and
$$\fN:=\{ \N \mid \N\leq \G \text{ closed normal subgroup scheme of }\G \}$$
given by
$\Psi:\fT\to \fN, T\mapsto \underline{\Aut}_{\C_T\text{-}\alg}(R)$ resp.~$\Phi:\fN\to \fT, \N\mapsto R^{\N}$.
\end{thm}

Here, the ring of invariants $R^\N$ is the largest subobject $T$ of $R$ such that for all $k$-algebras $D$ and all $\sigma\in \N(D)\subset \Aut_{\C_{\iota(D)}}(R\otimes_k D)$, one has $\sigma|_{T\otimes_k D}=\id_{T\otimes_k D}$. Equivalently, $R^\N$ is the equalizer of the 
morphisms $\id_R\otimes u_{k[\N]}:R\to R\otimes_k k[\N]$~\footnote{$k[\N]:=\O_\N(\N)$ denotes the ring of regular functions on the affine scheme $\N$.} and
$R\xrightarrow{\delta} R\otimes_k H\twoheadrightarrow R\otimes_k  k[\N]$, where $\delta=\eps_{R_R}^{-1}\circ (u_R\otimes \id_R)$ is the comodule map of $R$ as $H$-comodule, and $H\twoheadrightarrow k[\N]$ is the canonical epimorphism.

%

\begin{proof}[Proof of Thm.~\ref{thm:galois-correspondence}]
The functor $\omega_R$ is an equivalence of categories
$$\omega_R:\tenscat{M}\to \comod(H),$$
and also of their ind-categories.\footnote{Here, $\comod(H)$ denotes the category of left-$H$-comodules which are finite-dimensional as $k$-vector spaces.}
Hence, it provides a bijection between subalgebras of $R$ in $\C$ and subalgebras of $H$ stable under the left comodule structure.\\
We will show that under this bijection sub-PV-rings correspond to sub-Hopf-algebras and that this bijection can also be described as given in the theorem.

First, let $T\subseteq R$ be a PV-ring  for some $N\in \tenscat{M}$. Then $\tenscat{N}$ is a full subcategory of $\tenscat{M}$, and the fibre functor $\omega_T:\tenscat{N}\to \vect_k$ corresponding to $T$ is nothing else than the restriction of $\omega_R$ to the subcategory $\tenscat{N}$, as $T$ is a subobject of $R$.
Hence, $H':=\omega_R(T)=\omega_T(T)$ is a sub-Hopf-algebra of $H$.
Therefore, we obtain a closed normal subgroup scheme of $\G=\spec(H)$ as the kernel
of $\spec(H)\twoheadrightarrow \spec(H')$.
As $\spec(H)=\underline{\Aut}_{\C\text{-}\alg}(R)$ and $\spec(H')=\underline{\Aut}_{\C\text{-}\alg}(T)$, this kernel is exactly
$\underline{\Aut}_{\C_T\text{-}\alg}(R)$.

On the other hand, let $\N$ be a closed normal subgroup scheme of $\G=\spec(H)$ defined by a normal Hopf-ideal $I$ of $H$, and
$$H'=\Ker\left(H\xrightarrow{\Delta - \id_H\otimes u_H}H\otimes_k H\twoheadrightarrow H\otimes_k (H/I)\right)$$ 
the corresponding sub-Hopf-algebra of $H$.\\
The subcategory $\comod(H')$  is generated by one object $V$ (as every category of finite comodules is), and the object $N\in \tenscat{M}$ corresponding to $V$ via $\omega_R$, has a PV-ring $T$ inside $R$ by Thm.~\ref{thm:existence-of-pv-ring}, since $R$ is a simple solution ring for $N$ with $R^\C=k$. Furthermore, since $T$ is the PV-ring corresponding to the fibre functor $\omega_R:\tenscat{N}\to \comod(H')$, we have $\omega_R(T)=H'$.

It remains to show that $T=R^\N$, i.e. that
$$T=\Ker\left( R\xrightarrow{\delta\, -\,\, \id_R\otimes_k u_H}R\otimes_k H\twoheadrightarrow R\otimes_k k[\N]=R\otimes_k (H/I)\right).$$
As $\omega_R$ is an equivalence of categories, this is equivalent to
$$\omega_R(T)=\Ker\left( \omega_R(R)\xrightarrow{\omega_R(\delta) \, -\,\, \omega_R(\id_R)\otimes_k u_H}\omega_R(R)\otimes_k H\twoheadrightarrow \omega_R(R)\otimes_k (H/I)\right).$$
But, as $\omega_R(T)=H'$, $\omega_R(R)=H$ and $\omega_R(\delta)=\Delta$, this is just the definition of $H'$.
\end{proof}

\bibliographystyle{plain}
\def\cprime{$'$} 


\vspace*{.5cm}

\parindent0cm

\end{document}